\documentclass[11pt,reqno]{amsart}

\usepackage[text={155mm,230mm},centering]{geometry}   

\usepackage{graphicx}

\usepackage{epstopdf}
\usepackage{subfigure}
\usepackage{latexsym,bm}

\usepackage{amsmath,amsthm,amssymb,amsfonts}

\usepackage[mathscr]{eucal}

\usepackage[all]{xy}

\usepackage{mathrsfs}

\usepackage{hyperref}

\newtheorem{theorem}{Theorem}[section]
\newtheorem{lemma}[theorem]{Lemma}

\newtheorem{corollary}[theorem]{Corollary}
\newtheorem{proposition}[theorem]{Proposition}

\theoremstyle{definition}
\newtheorem{example}[theorem]{Example}
\newtheorem{definition}[theorem]{Definition}
\newtheorem{definition-lemma}[theorem]{Definition-Lemma}
\newtheorem{definition-theorem}[theorem]{Definition-Theorem}

\newtheorem{remark}[theorem]{Remark}

\newtheorem{convention}[theorem]{Convention}

\newtheorem*{ack}{Acknowledgements}

\newcommand{\ac}{\textup{!`}}

\newcommand{\ldb}{\mathopen{\{\!\!\{}}
\newcommand{\rdb}{\mathclose{\}\!\!\}}}
\newcommand{\ldbg}{\mathopen{\bigl\{\!\!\bigl\{}}
\newcommand{\rdbg}{\mathclose{\bigr\}\!\!\bigr\}}}

\allowdisplaybreaks

\title[Calabi-Yau algebras]
{Calabi-Yau algebras and the shifted noncommutative symplectic structure}

\author{Xiaojun Chen}

\address{Department of Mathematics, Sichuan University, Chengdu, Sichuan Province, 610064 P. R. China}
\email{xjchen@scu.edu.cn}

\author{Farkhod Eshmatov}

\address{Beijing Advanced Innovation Center for Imaging Theory and Technology,
Capital Normal University, Beijing, 100048 P. R. China}
\email{olimjon55@hotmail.com}

\date{February 8, 2018}

\begin{document}

\begin{abstract}
In this paper we show that for a Koszul Calabi-Yau algebra,
there is a shifted bi-symplectic structure in the sense 
of Crawley-Boevey-Etingof-Ginzburg \cite{CBEG}, 
on the cobar construction of its co-unitalized Koszul dual coalgebra,
and hence its DG representation schemes,
in the sense of Berest-Khachatryan-Ramadoss \cite{BKR}, have a shifted symplectic structure
in the sense of Pantev-To\"en-Vaqui\'e-Vezzosi \cite{PTVV}.

\noindent{\bf Keywords:} Shifted noncommutative symplectic structure, Calabi-Yau algebra

\end{abstract}

\maketitle

\setcounter{tocdepth}{1}\tableofcontents


\section{Introduction}\label{sect:intro}


The notion of Calabi-Yau algebras was introduced by Ginzburg in 2007. It is a noncommutative
generalization of affine Calabi-Yau varieties, and due to Van den Bergh, admits the so-called
``noncommutative Poincar\'e duality". In a joint paper
with Berest and Ramadoss \cite{BCER} we showed that for a Koszul Calabi-Yau
algebra, say $A$, there is a version of noncommutative Poisson structure, called
the {\it shifted double Poisson structure}, on its cofibrant resolution,
and hence
induces a shifted Poisson structure on the {\it derived representation schemes} of $A$,
a notion introduced by Berest, Khachatryan and Ramadoss in \cite{BKR}.
Here ``shifted" means there is a degree shifting, depending on the dimension of
the Calabi-Yau algebra, on the Poisson bracket.
Such a shifted Poisson structure was later further studied in \cite{CEEY} in much detail.

A natural question that arises now is, given a Koszul Calabi-Yau algebra, whether
or not the noncommutative Poisson
structure associated to it is noncommutative symplectic. To justify this question,
let us remind the reader of a version of noncommutative symplectic structure introduced by
Crawley-Boevey, Etingof and Ginzburg in \cite{CBEG}, which they called
the {\it bi-symplectic structure} in the paper.
A bi-symplectic structure on an associative algebra is
a closed 2-form in its Karoubi-de Rham complex
that induces an isomorphism between the space of noncommutative vector fields
(more precisely, the space of {\it double derivations}) and the space
of noncommutative 1-forms.
In \cite[Appendix]{VdB}, Van den Bergh showed that any bi-symplectic structure
naturally
gives rise to a double Poisson structure, which is completely analogous to the classical case.
However, the reverse is in general not true. Nevertheless, it is still interesting to ask
if this is true in the special case of Calabi-Yau algebras.

The second motivation for the paper comes from the 2012 paper \cite{PTVV}
of Pantev, To\"en, Vaqui\'e and Vezzosi, where they introduced 
the notion of {\it shifted symplectic structure} for derived stacks.
This not only generalizes classical symplectic geometry
to a much broader context, but also reveals many new features
on a lot of geometric spaces, especially on various moduli spaces that 
mathematicians are now studying.
As remarked by the authors, the shifted symplectic structure, if it exists,
always comes from the Poincar\'e duality of the corresponding source spaces;
they also outlined how to generalize shifted symplectic structures
to noncommutative spaces, such as Calabi-Yau categories.

Note that Calabi-Yau algebras are highly related to Calabi-Yau categories. 
For example, a theorem of Keller (see \cite[Lemma 4.1]{Keller1}) 
says that the bounded derived category of a Calabi-Yau algebra 
is a Calabi-Yau category. 
Applying the idea of \cite{PTVV}, we would expect that the noncommutative Poincar\'e duality
of a Calabi-Yau algebra
shall also play a role in the corresponding shifted noncommutative symplectic structure if it exists.

The main results of the current paper may be summarized as follows.
Let $A$ be a Calabi-Yau algebra of dimension $n$. 
Assume that $A$ is also Koszul, and denote its Koszul dual coalegbra by $A^{\ac}$.
Let $\tilde R=\Omega(\tilde A^{\ac})$ be the cobar construction of the co-unitalization $\tilde A^{\ac}$
of $A^{\ac}$.
In this paper, we show that $\tilde R$, 
rather than $\Omega(A^{\ac})$ as studied in \cite{BCER,CEEY}, 
has a $(2-n)$-shifted bi-symplectic structure.
Such a bi-symplectic structure
comes from the volume form of the noncommutative Poincar\'e duality of $A$,
and naturally induces a $(2-n)$-shifted symplectic structure on the 
DG representation schemes $\mathrm{Rep}_V(\tilde R)$, for all vector spaces
$V$ (see Theorem \ref{thm:mainthm1}).
By taking the corresponding trace maps we obtain a commutative diagram
(see \eqref{diag:tracefromHochschildcomplex} for more details)
\begin{equation}\label{digram:PDandshiftedsymplectic}
\xymatrixcolsep{4pc}
\xymatrix{
\mathrm{HH}^\bullet(A)\ar[r]^-{\mathrm{Tr}}\ar[d]^{\cong}
&\mathrm H^\bullet(\mathrm{Der}(\mathrm{Rep}_V(\tilde R))^{\mathrm{GL}})\ar[d]^{\cong}\\
\mathrm{HH}_{n-\bullet}(A)\ar[r]^-{\mathrm{Tr}}
&\mathrm H_{n-\bullet}(\Omega^1_{\mathrm{com}}(\mathrm{Rep}_V(\tilde R))^{\mathrm{GL}}),
}
\end{equation}
where the left hand side are the Hochschild cohomology and homology
of $A$, with the isomorphism
being Van den Bergh's noncommutative Poincar\'e duality, 
and the right hand side are the cohomology and homology of
the $\mathrm{GL}(V)$-invariant
complexes of vector fields and 1-forms on $\mathrm{Rep}_V(\tilde R)$ respectively.

By Van den Bergh's result mentioned above, the shifted bi-symplectic structure
induces a
shifted double Poisson structure on $\tilde R$; therefore there is a shifted Poisson
structure on $\mathrm{Rep}_V(\tilde R)$.
In this paper, we will study the deformation
quantization of such a shifted Poisson structure,
and show that it comes from the quantization of the ``functions"  
$\tilde R_\natural:=\tilde R/[\tilde R,\tilde R]$ of $\tilde R$ 
(see Theorem \ref{thm:existenceofquantization}). By Koszul duality,
the homology of $\tilde R_{\natural}$ minus the unit
is isomorphic to the cyclic homology of $A$, 
and thus we obtain a quantization of the latter as well.
This construction 
is inspired by the quantization of quiver representations.

This paper is organized as follows.
In \S\ref{sect:basicsofNCgeometry} we collect some basics of noncommutative geometry,
such as the noncommutative 1-forms and vector fields, and some operations, such as the contraction
and Lie derivative between them, then we recall the definition of bi-symplectic structure introduced by Crawley-Boevey 
et. al. in \cite{CBEG}.

In \S\ref{sect:Koszulduality} we recall the notion of Koszul algebras and some of their basic properties.
Let $A$ be a Koszul algebra over a field $k$. 
We give explicit formulas for the double derivations and 1-forms of $\tilde R$.
The commutator quotient spaces of them are identified with the Hochschild cohomology
and homology of $A$ respectively.

In \S\ref{sect:KoszulCY} we first recall the definition of Calabi-Yau algebras, 
and then show that if the Calabi-Yau algebra, say $A$, is Koszul, then its
noncommutative volume form gives a shifted bi-symplectic structure
on $\tilde R$, where $\tilde R$ is given as above.

In \S\ref{sect:defofsss} we first recall the DG representation schemes
of a DG algebra, which was introduced by Berest et. al. in \cite{BKR}.
Following the works \cite{BKR,CBEG}, we see that
if a DG algebra admits a shifted
bi-symplectic structure, then its DG representation schemes have a shifted
symplectic structure. We then apply it to the Koszul Calabi-Yau algebra case.

In \S\ref{sect:shiftedPoisson} and \S\ref{sect:quantization}
we study the shifted Poisson
structure on the representation schemes of $\tilde R$ and their quantization. 
We show that such quantizations
are induced by the quantization of $\tilde R_{\natural}$ as a Lie bialgebra. 
This is completely analogous to the papers of Schedler \cite{Schedler}
and Ginzburg-Schedler \cite{GS}, where the quantization
of the representation spaces of doubled quivers is constructed, which is
compatible with the quantization of the necklace Lie bialgebra of the quivers via the
canonical trace map.

In \S\ref{sect:DRep}, we briefly discuss some relationships
of the current paper with the series of papers by Berest and his collaborators
\cite{BCER,BFR,BKR,BR}, where the derived representation schemes
of associative algebras were introduced and studied.

In the last section, \S\ref{sect:example}, we give the two examples of Calabi-Yau alegebras,
namely, the 3- and 4-dimensional Sklyanin algebras, and study the corresponding
shifted bi-symplectic structure in some detail.

This paper is a sequel to \cite{BCER,CEEY}, where
the shifted noncommutative Poisson structure associated to Calabi-Yau algebras
was studied; however, in the current paper we try to
be as self-contained as possible.
When we were in the final stage of the paper, we learned that 
Y. T. Lam in his thesis \cite{Lam}
as well as W.-K. Yeung in his paper \cite{Yeung}
have obtained several results which are similar to ours.
Nevertheless, the main results and methods of theirs
and ours are quite different.

\begin{convention}
Throughout the paper, $k$ is a field of characteristic zero, though in a lot of cases it need not
necessarily to be so. All morphisms and 
tensors are over $k$ unless otherwise specified. 
DG algebras (respectively DG coalgebras) are unital and
augmented (respectively co-unital and co-augmented), with
the degree of the differential being $-1$.
For a chain complex, its homology is denoted by $\mathrm H_\bullet(-)$, and its cohomology
is given by $\mathrm H^\bullet(-):=\mathrm H_{-\bullet}(-)$.
\end{convention}

\section{Some basics of noncommutative geometry}\label{sect:basicsofNCgeometry}

In this section, we recall some basic notions in noncommutative geometry.
They are mostly taken from \cite{CBEG,VdB,VdB2}; 
here we work in the differential graded (DG for short) setting.

\subsection{Noncommutative differential forms}

Suppose that $(R, \partial)$ is a DG associative algebra over $k$,
where $\partial$ is the differential.
The bimodule $\Omega^1_{\mathrm{nc}}R$ of {\it noncommutative 1-forms} of $R$ is
a DG $R$-bimodule
generated by symbols $d x$ of degree $|x|$, linear in $x$ for all $x\in R$ and 
subject to the relations
\begin{eqnarray*}
d  (xy) &=& (d  x)y+x (d  y),\\
\partial(dx)&=&d(\partial x),
\end{eqnarray*}
for all $x, y\in R$.
In this paper, we also assume $d 1 =0$.

Alternatively,
$\Omega^1_{\mathrm{nc}}R$ is
the kernel of the multiplication $R\otimes R\to R$,
which is a subcomplex
of $R$-bimodules in $R\otimes R$, 
generated by $1\otimes x-x\otimes 1$ for all
$x\in R$, with differential induced from $\partial$.
The identification of two $R$-bimodules constructed above is given by
$dx\mapsto 1\otimes x-x\otimes 1$. 

Let $\Omega^1_{\mathrm{nc}}R[-1]$ be 
the suspension of $\Omega^1_{\mathrm{nc}}R$, i.e.
the degrees of $\Omega^1_{\mathrm{nc}}R$ are shifted up by one.
Sometimes we also write it in the form $\Sigma\Omega^1_{\mathrm{nc}}R$,
which means the degree-shifting operator applies from the left. 
(In what follows, for a graded vector space $V_\bullet$,
$V[n]$ is a graded vector space with $(V[n])_i=V_{i+n}$.)

Let $\Omega^\bullet_{\mathrm{nc}}R:=T_R(\Omega^1_{\mathrm{nc}} R[-1])$
be the free tensor algebra generated by $\Omega^1_{\mathrm{nc}}R[-1]$ over $R$.
Let
$$
d:R\to \Omega^1_{\mathrm{nc}}R[-1],\;
x\mapsto d(x)= \Sigma dx.
$$
Then by our sign convention, it is easy to check\footnote{Recall that
for a graded $R$-bimodule $M$, $\Sigma M$ is a graded $R$-bimodule with
$a\cdot \Sigma m\cdot b=(-1)^{|a|}\Sigma (amb)$, for homogeneous $a, b\in R$ and $m\in M$.}
\begin{eqnarray*}
d(xy)&=&d(x) \cdot y+(-1)^{|x|}x\cdot d(y),\\
d\circ \partial(x)&=&-\partial \circ d(x),
\end{eqnarray*}
for all $x, y\in R$.

Now let $d(d(x))=0$ for all $x\in R$.
Extend $d$ to be a map $d:\Omega^\bullet_{\mathrm{nc}}R\to \Omega^{\bullet+1}_{\mathrm{nc}}R$
by derivation.
Also, extend $\partial$ to $\Omega^\bullet_{\mathrm{nc}}R$ by derivation.
Then we have 
$$d^2=0,\quad \partial^2=0,\quad\mbox{and}\quad d\circ \partial+\partial \circ d=0.$$
In general, $d$ is called the {\it de Rham differential} of $\Omega^\bullet_{\mathrm{nc}}R$,
and for convenience, $\partial$ is called the {\it internal differential} of $\Omega^\bullet_{\mathrm{nc}}R$.

Note that $\partial$ and $d$ have degrees $-1$ and $1$ respectively,
which make
$\Omega^\bullet_{\mathrm{nc}}R$ 
into a mixed DG algebra and
is called the set of {\it noncommutative
differential forms} of $R$.
Let
$$\mathrm{DR}^\bullet_{\mathrm{nc}}R:=\big(\Omega^\bullet_{\mathrm{nc}}R\big)_{\natural}
=\Omega^\bullet_{\mathrm{nc}}R/[\Omega^\bullet_{\mathrm{nc}}R,\Omega^\bullet_{\mathrm{nc}}R]$$
be the graded commutator quotient space.
Since $\partial$ and $d$
are both derivations with respect to the product on $\Omega^\bullet_{\mathrm{nc}}R$,
$\mathrm{DR}^\bullet_{\mathrm{nc}}R$ with the induced differentials, still denoted by $(\partial,d)$, is 
a mixed complex, and
is called the {\it Karoubi-de Rham complex} of $R$.

\subsection{Noncommutative polyvector fields}

Following \cite{CBEG,VdB}, the noncommutative vector fields
on an associative algebra are given by the {\it double derivations}.
By definition,
the space of double derivations $\mathbb D\mathrm{er}\; R$ of $R$
is the set of derivations
$\mathrm{Der}(R, R\otimes R)$.
Since the map $R\to \Omega^1_{\mathrm{nc}}R, x\mapsto dx=1\otimes x-x\otimes 1$
is a universal derivation, meaning that every derivation of $R$ factors through 
$\Omega^1_{\mathrm{nc}}R$ ({\it c.f.} \cite[Proposition 3.3]{Quillen}),
we have that
$$
\mathbb D\mathrm{er}\; R\cong\mathrm{Hom}_{R^e}(\Omega^1_{\mathrm{nc}} R, R\otimes R),
$$
where $R^e$ is the enveloping algebra $R\otimes R^{\mathrm{op}}$.
In the above notation, we have used the {\it outer $R$-bimodule} structure on
$R\otimes R$; namely, for any $x, y, u, v\in R$,
$$
u\cdot (x\otimes y)\cdot v:=ux\otimes yv.
$$
$R\otimes R$ has also an {\it inner $R$-bimoule} structure, which is given
by
$$
u * (x\otimes y)* v:=(-1)^{|u||x|+|v||y|+|u||v|}xv\otimes uy.
$$
With the inner $R$-bimodule structure on $R\otimes R$,
$\mathbb D\mathrm{er}\; R$
is a DG $R$-bimodule.
Now let
$
T_R(\mathbb D\mathrm{er}\; R[1])
$
be the free DG algebra generated by $\mathbb D\mathrm{er}\; R[1]$
over $R$, which is called the space of {\it (noncommutative) polyvector fields} of $R$.

\subsection{Actions of noncommutative vectors on noncommutative forms}

Analogous to the classical case, the noncommutative vectors act on
noncommutative forms by contraction and by Lie derivative, which together
satisfy the noncommutative version of the Cartan identity (see Lemma \ref{Cartanformula}).

\subsubsection{Contraction and Lie derivative}

For any $\Theta\in\mathbb D\mathrm{er}\; R[1]$, 
the {\it contraction} operator
$$
i_{\Theta}:\Omega^\bullet_{\mathrm{nc}}R\longrightarrow
\Omega^{\bullet}_{\mathrm{nc}}R\otimes
\Omega^{\bullet}_{\mathrm{nc}}R.$$
is given as follows: first,
let
$$i_{\Theta}(a)=0,\quad\mbox{for all}\; a\in R$$
and
$$
i_{\Theta}: \Omega^1_{\mathrm{nc}} R[-1]\to  R\otimes R,\; 
d(a)\mapsto \Theta(a),
$$
where we have used the fact that
$\mathbb D\mathrm{er}\; R[1]=\mathrm{Hom}_{R^e}(\Omega^{1}_{\mathrm{nc}} R[-1], R\otimes R).$
Second, extend $i_{\Theta}$ to an $R$-linear operator on 
$\Omega^\bullet_{\mathrm{nc}}R$ in the natural way,
that is,
if we write
$$i_{\Theta} a=\sum i_{\Theta}'a\otimes i_{\Theta}'' a, \quad\mbox{for}\; a\in\Omega^1_{\mathrm{nc}} R,$$
then
$$
i_{\Theta}(a_1\cdots a_n)=\sum_{1\le k\le n}(-1)^{\sigma_k}
a_1\cdots a_{k-1}(i_{\Theta}'a_k)\otimes (i_{\Theta}'' a_k) a_{k+1}\cdots a_n,
$$
for $a_1,\cdots, a_n\in \Omega^1_{\mathrm{nc}} R$,
where $(-1)^{\sigma_k}$ is the Koszul sign.
Recall that in general, the Koszul sign comes from switching the positions of graded elements:
for two graded vector spaces $V, W$, the isomorphism
$V\otimes W\to W\otimes V$ is given by $v\otimes w\mapsto (-1)^{|v||w|}w\otimes v$;
for example, in the formula above, $(-1)^{\sigma_k}=(-1)^{|\Theta|(|a_1|+\cdots+|a_{k-1}|)}$.
In what follows, we will sometimes omit its precise value if it is clear from the context.

Next, given $\Theta\in\mathbb D\mathrm{er}\; R[1]$
the {\it Lie derivative} of $\Theta$ on the noncommutative differential forms
is given by
\begin{eqnarray*}
L_{\Theta}(x_0dx_1\cdots dx_n)
&:=&
\Theta'(x_0)\otimes\Theta''(x_0)dx_1\cdots dx_n\\
&&+\sum_{1\le k\le n}(-1)^{\sigma_k}\big(x_0dx_1\cdots dx_{k-1}d\Theta'(x_k)\otimes\Theta''(x_k)dx_{k+1}\cdots dx_n\\
&&\quad\quad+(-1)^{|\Theta'(x_k)|}x_0dx_1\cdots dx_{k-1} \Theta'(x_k)\otimes d\Theta''(x_k) dx_{k+1}\cdots dx_n\big).
\end{eqnarray*}

\subsubsection{Reduced contraction and Lie derivative}

Now, for a graded algebra $A$ and $c=c_1\otimes c_2\in A\otimes A$,
let
$^\circ c=(-1)^{|c_1||c_2|}c_2c_1$,
and given a map $\phi: A\to A\otimes A$, write
$^\circ\phi: A\to A: c\mapsto\;\!^\circ(\phi(c))$.

Define the {\it reduced contraction operator} $\iota$ 
and the {\it reduced Lie derivative} $\mathscr L$ by
\begin{equation}\label{formula:reducedcontraction}
\iota_{\Theta}(-): \Omega^\bullet_{\mathrm{nc}}R \to  \Omega^{\bullet-1}_{\mathrm{nc}}R,\;
 \omega \mapsto \iota_{\Theta}\omega=\;\!^{\circ}(i_{\Theta}(\omega))
\end{equation}
and
\begin{equation}
\mathscr L_{\Theta}(-): \Omega^\bullet_{\mathrm{nc}}R \to  \Omega^{\bullet}_{\mathrm{nc}}R,\;
 \omega \mapsto \mathscr L_{\Theta}\omega=\;\!^\circ(L_{\Theta}(\omega))
\end{equation}
respectively.
More explicitly,
for $a_1, a_2,\cdots, a_q\in\Omega^1_{\mathrm{nc}} R$,
\begin{equation*}
\iota_{\Theta}(a_1a_2\cdots a_r)=\sum_{k}(-1)^{\sigma_k}
(i_{\Theta}''a_k)a_{k+1}\cdots a_ra_1\cdots a_{k-1}(i_{\Theta}'a_k),
\end{equation*}
and similarly for $\mathscr L_{\Theta}(a_1a_2\cdots a_r)$
(see \cite[(2.8.4) and (2.8.5)]{CBEG}).
The following two lemmas are proved in \cite{CBEG}.

\begin{lemma}[\cite{CBEG} Lemma 2.8.6]
\begin{enumerate}
\item[$(1)$] The reduced contraction defined above only
depends on the image of $\omega$
in $\mathrm{DR}_{\mathrm{nc}}^\bullet R $;
in other words, $\iota_{\Theta}$ descends to a well-defined map
$\mathrm{DR}_{\mathrm{nc}}^\bullet R \to  \Omega^{\bullet-1}_{\mathrm{nc}}R$.
\item[$(2)$] For $\omega\in\mathrm{DR}^n_{\mathrm{nc}} R $,
the map
$\Theta\mapsto\iota_{\Theta}\omega$
gives an $R$-bimodule morphism
$
\mathbb D\mathrm{er}\; R[1]\to\Omega^{n-1}_{\mathrm{nc}} R 
$.
\end{enumerate}
\end{lemma}

\begin{lemma}[\cite{CBEG} Lemma 2.8.8(i)]\label{Cartanformula}
Let $\iota$ and $\mathscr L$ be as above. Then for any $\Theta\in\mathbb D\mathrm{er}\; R[1]$,
we have 
$$
d\circ \iota_{\Theta}+\iota_{\Theta}\circ d=\mathscr L_{\Theta}\quad\mbox{and}\quad
d\circ\mathscr L_{\Theta}=\mathscr L_{\Theta}\circ d.
$$
\end{lemma}


Moreover, from the definitions of $\iota$ and $\mathscr L$, it is clear
that both of them respect $\partial$. 
With these preparations, we now recall the definition of {\it bi-symplectic structures} 
of Crawley-Boevey, Etingof and Ginzburg introduced in \cite{CBEG} (we here rephrase it
for DG algebras):

\begin{definition}[Bi-symplectic structure]
Suppose $R$ is a DG associative algebra.
A closed form $\omega\in\mathrm{DR}^2_{\mathrm{nc}} R $ of total degree $2-n$
is called an {\it $n$-shifted bi-symplectic structure}
if the
map
\begin{equation}\label{isofrombisymplectic}
\iota_{(-)}\omega: \mathbb{D}\mathrm{er}\; R[1] \to (\Omega^1_{\mathrm{nc}}R[-1])[2-n],\;
\Theta \mapsto \iota_{\Theta}\omega
\end{equation}
is a quasi-isomorphism of complexes of $R$-bimodules.
\end{definition}

The reader may refer to Remark \ref{rmk:degreeshifting}
for some discussions on degree shifting in the above definition.

In general, it is difficult to check the existence
of a bi-symplectic structure for a given DG algebra. The difficulty 
lies in the fact there is in general no closed formula for the noncommutative
vector fields and 1-forms. 
However, for free
associative algebras, such as the path algebra of a quiver, both of
them can be explicitly written down. This leads Crawley-Boevey et. al. \cite{CBEG}
to give an explicit identification of  them, and hence to give a bi-symplectic
structure on the path algebra of a doubled quiver. 
In the following two sections, we generalize their construction to the DG case, where Koszul
Calabi-Yau algebras naturally appear.

\section{Koszul duality}\label{sect:Koszulduality}

Koszul duality was originally introduced by Priddy to
compute the Hochschild
homology and cohomology of associative algebras.
Nowadays it plays an increasingly important role in the study of
noncommutative algebraic geometry.
Let $A$ be a Koszul algebra, and $A^{\ac}$ its Koszul dual coalgebra.
Let $\tilde R$ be the cobar construction of $\tilde A^{\ac}$, where $\tilde A^{\ac}$ is the co-unitalization
of $A^{\ac}$. In this section, we give explicit formulas for $\Omega^1_{\mathrm{nc}}\tilde R$
and $\mathbb D\mathrm{er}\;\tilde R$, which are very much related to the
Hochschild homology and cohomology of $A$
(see Propositions \ref{Prop:cotangentcomplex}
and \ref{Prop:tangentcomplex} below).

\subsection{Koszul algebra}\label{subsect:Koszul}

Let $W$ be a finite-dimensional vector space over $k$.
Denote by $TW$ the free algebra generated by $W$ over $k$.
Let $Q$ be a subspace of $W\otimes W$, and let
$(Q)$ be the two-sided ideal generated by $Q$ in $TW$.
Then the quotient algebra
$A:= TW/(Q)$
is called
a {\it quadratic algebra}.

Consider the subspace
$$U=\bigoplus_{n=0}^\infty U_n:=
\bigoplus_{n=0}^\infty \bigcap_{i+j+2=n}W^{\otimes i}\otimes Q\otimes
W^{\otimes j}$$
of $TW$,
then $U$ is a coalgebra whose coproduct is induced from
the de-concatenation of the tensor products.
The {\it Koszul dual coalgebra} of $A$, denoted
by $A^{\ac}$, is
$$
A^{\ac}=\bigoplus_{n=0}^\infty \Sigma^{\otimes n} (U_n).
$$
$A^{\ac}$ has a graded coalgebra structure induced from that of $U$ with
$$
(A^{\ac})_0=k, \quad (A^{\ac})_1=\Sigma W, \quad (A^{\ac})_2=(\Sigma\otimes\Sigma)(Q),\quad\cdots\cdots
$$

The {\it Koszul dual algebra} of $A$, denoted by $A^!$,
is just the linear dual space of $A^{\ac}$, which is then a graded algebra.
More precisely,
let $W^*=\mathrm{Hom}(W, k)$ be the linear dual space of $W$,
and let $Q^\perp$ denote
the space of annihilators of $Q$ in $W^*\otimes W^*$.
Shift the grading of $W^*$ down by one, denoted by $\Sigma^{-1}W^*$,
then
$$
A^!=T(\Sigma^{-1}W^*)/(\Sigma^{-1}\otimes\Sigma^{-1}\circ Q^{\perp}).
$$

Choose a basis $\{e_i\}$ for $W$, and let $\{e_i^*\}$ be their duals in $W^*$.
Let $\{\xi_i\}$ be the basis in $\Sigma^{-1}W^*$ corresponding to $\{e_i^*\}$, i.e.
$\xi_i=\Sigma^{-1}e_i^*$.
There is a chain complex associated to $A$, called the {\it Koszul complex}:
\begin{equation}\label{Koszul_complex}
\xymatrix{
\cdots\ar[r]^-{b'}&
A\otimes A^{\ac}_{i+1}\ar[r]^-{b'}&
A\otimes A^{\ac}_{i}\ar[r]^-{b'}&
\cdots\ar[r]&
A\otimes A^{\ac}_0\ar[r]^-{b'}& k,
}
\end{equation}
where for any $r\otimes u\in A\otimes A^{\ac}$,
$b'(r\otimes u)=\sum_i re_i\otimes\xi_i\vdash u$.
Here $\xi_i\vdash u$ means the interior product (contraction) of $\xi_i$
with $u$, or in other words, the evaluation of $\xi_i$ on the first component
of $u$. In what follows, we prefer to write it in the form $u\cdot \xi_i$
or simply  $u\xi_i$, since the interior product gives a right $A^!$-module structure
on $A^{\ac}$.
(In what follows we will also use the contraction from the right,
and in this case we write it in the form $\xi_i u$ for $\xi_i$ and $u$ as above.)

\begin{definition}[Koszul algebra]
A quadratic algebra $A=TW/(Q)$ is called {\it Koszul}
if the Koszul complex \eqref{Koszul_complex} is acyclic.
\end{definition}

Applying the graded version of Nakayama Lemma, we have 
the following (see \cite[Proposition 3.1]{VdB-1} for a proof):

\begin{proposition}\label{KoszulResolutionOfA}
Suppose $A$ is a Koszul algebra. 
Then
the following complex
\begin{equation}\label{Comp:KoszulResolutionOfA}
\xymatrix{
\cdots\ar[r]&A\otimes A^{\ac}_m\otimes A\ar[r]^-b &A\otimes A^{\ac}_{m-1}\otimes A
\ar[r]^-b&
A\otimes A^{\ac}_0\otimes A\cong A\otimes A \ar[r]^-{\mu} &A,
}
\end{equation}
where
\begin{equation}
b(a\otimes c\otimes a')=
\sum_i \Big(ae_i\otimes  c\xi_i\otimes a'+(-1)^{m}a\otimes \xi_i c  \otimes e_i a'\Big)
\end{equation}
for $a\otimes c\otimes a' \in A\otimes A^{\ac}_m\otimes A$, and $\mu $ is the multiplication
on $A$,
gives a resolution of $A$ as an $A^e$-module.
\end{proposition}

Let $K(A):=(A\otimes A^{\ac}\otimes A, b)$. Then
\begin{equation}
\mathrm{Tor}_\bullet^{A^e}(A, A)=\mathrm{H}_\bullet(A\otimes_{A^e}K(A))
\quad\mbox{and}\quad
\mathrm{Ext}^\bullet_{A^e}(A, A)=\mathrm{H}^\bullet(\mathrm{Hom}_{A^e}(K(A), A)),
\end{equation}
which are also identical to the Hochschild 
homology $\mathrm{HH}_\bullet(A)$ and cohomology $\mathrm{HH}^\bullet(A)$ of $A$ respectively.
This result is due to Priddy:

\begin{proposition}[Priddy \cite{Priddy}]\label{lemma:identityofHochschild}
The complexes $A\otimes_{A^e}K(A)$
and
$\mathrm{Hom}_{A^e}(K(A), A)$
are quasi-isomorphic to the Hochschild chain complex
$\mathrm{CH}_\bullet(A)$
and the Hochschild cochain complexes $\mathrm{CH}^\bullet(A)$ respectively.
\end{proposition}

There are explicit formulas for the chain complexes
$A\otimes_{A^e}K(A)$
and
$\mathrm{Hom}_{A^e}(K(A), A)$: as graded vector
spaces, they are the same as $A\otimes A^{\ac}_\bullet$ and $A\otimes A^!_{\bullet}$, while
the differentials are given by
\begin{equation}\label{DifferentialInKoszulComplex}
b(a\otimes u)=\sum_{i}
\Big(a e_i \otimes u \xi_i  +(-1)^{|a|}  e_i a\otimes \xi_i u\Big),
\end{equation}
and 
\begin{equation}\label{DifferentialInComplexComputingHochschildCohomology}
b(a\otimes x)=\sum_i\Big(a e_i \otimes x \xi_i  +(-1)^{|a|}  e_i a\otimes \xi_i x\Big),
\end{equation}
respectively. 
Van den Bergh \cite[Proposition 3.3]{VdB-1} gave a formula for
the quasi-isomorphism of these complexes.

Now for a Koszul algebra $A$, view it as a DG algebra concentrated in degree zero with
trivial differential. 
Let $A^{\ac}$ be the Koszul dual coalgebra of $A$,
and $\Omega(A^{\ac})$ be the cobar construction of $A^{\ac}$, 
whose differential is denoted by $\partial$. Then
from \eqref{Comp:KoszulResolutionOfA}
one can deduce that the natural surjective map
\begin{equation}\label{Koszulresolution}
(\Omega(A^{\ac}),\partial)\to (A, 0),
\end{equation}
is a quasi-isomorphism of DG algebras ({\it c.f.} \cite[Theorem 3.4.4]{LV}),
which then gives a cofibrant resolution of $A$ in the category
of DG associative algebras.

\subsection{Noncommutative geometry for Koszul algebras}

Suppose $A$ is a Koszul algebra.
Let $A^{\ac}$ and $A^!$ be its Koszul dual coalgebra and algebra respectively.
Let $\tilde A^{\ac}$ be the co-unitalization of $A^{\ac}$; that is, $\tilde A^{\ac}=A^{\ac}\oplus k$
with the coproduct $\tilde \Delta$ being
\begin{eqnarray*}
\tilde\Delta(a)&=&\Delta(a)+a\otimes 1+1\otimes a,\; \mbox{for}\; a\in A^{\ac}, \\
\tilde\Delta(1)&=&1\otimes 1,
\end{eqnarray*}
where 
$\Delta(a)$ is the coproduct of $a$ in $A^{\ac}$.
Let 
$\tilde R=\Omega(\tilde A^{\ac})$
be the cobar construction of $A^{\ac}$.
The following is straightforward since $\tilde R$ is a quasi-free DG algebra:

\begin{proposition}
Let $\tilde R=\Omega(\tilde A^{\ac})$.
Then
\begin{equation}
\Omega^1_{\mathrm{nc}}\tilde R[-1]\cong \tilde R\otimes A^{\ac}\otimes \tilde R \quad
\mbox{and}\quad
\mathbb D\mathrm{er}\;\tilde R[1]\cong \tilde R\otimes A^{!}\otimes \tilde R.
\end{equation}
\end{proposition}

\begin{proof}
In fact, since $\tilde R=\Omega(\tilde A^{\ac})=T(A^{\ac}[1])$ is quasi-free, 
we have the short exact sequence
$$0\longrightarrow \tilde R\otimes A^{\ac}[1]\otimes \tilde R\longrightarrow
\tilde R\otimes \tilde R\longrightarrow \tilde R\longrightarrow 0.$$
It follows that
$\Omega^1_{\mathrm{nc}} \tilde R=\tilde R\otimes A^{\ac}[1]\otimes \tilde R$,
that is, $\Omega^{1}_{\mathrm{nc}}\tilde R[-1]=\tilde R\otimes A^{\ac}\otimes \tilde R$.


From this, we also see that
\begin{eqnarray*}
\mathbb D\mathrm{er}\;\tilde R[1]&=&\mathrm{Hom}_{{\tilde R}^e}(\Omega^1_{\mathrm{nc}}\tilde R[-1], \tilde R\otimes\tilde R)\\
&=&\mathrm{Hom}_{{\tilde R}^e}(\tilde R\otimes A^{\ac}\otimes \tilde R, \tilde R\otimes\tilde R)\\
&=&\mathrm{Hom}(A^{\ac},\tilde R\otimes \tilde R)\\
&=&\tilde R\otimes A^!\otimes \tilde R,
\end{eqnarray*}
where in the last equality, we have identified $A^!\otimes (\tilde R\otimes\tilde R)$ 
with $\tilde R\otimes A^!\otimes
\tilde R$ via $a\otimes (u\otimes v)\mapsto (-1)^{(|a|+|u|)|v|}v\otimes a\otimes u$.
Thus $\mathbb D\mathrm{er}\;\tilde R[1]\cong \tilde R\otimes A^{!}\otimes \tilde R$ follows.
\end{proof}

The following proposition was obtained in \cite{BKR} (see also \cite{CEEY}),
and therefore we will only sketch its proof.
Given an associative algebra $A$,
denote by $(\mathrm{CC}_\bullet(A),b)$
and $(\mathrm{CH}_\bullet(A), b)$
the Connes cyclic complex and the Hochschild chain complex
of $A$ respectively.

\begin{proposition}[\cite{BKR}]\label{Prop:cotangentcomplex}
Suppose $A$ is a Koszul algebra, and $\tilde R=\Omega(\tilde A^{\ac})$ with differential $\partial$.
Then
\begin{equation}\label{qiscyclicandhoch}
(\bar{\tilde R}_\natural,\partial)\simeq(\mathrm{CC}_\bullet(A)[1],b)
\quad\mbox{and}
\quad(\mathrm{DR}^1_{\mathrm{nc}} \tilde R,\partial)\simeq(\mathrm{CH}_\bullet(A),b)
\end{equation}
as chain complexes, where $\bar{\tilde R}$ is the augmentation ideal of $\tilde R$.
\end{proposition}

To prove this proposition, we have to recall the definition
of the cyclic homology of coalgebras.
Suppose that $C$ is a DG coalgebra, let $\Omega(\tilde C)$ be the
cobar construction of the co-unitalization of $C$.
Let $\bar{\Omega}(\tilde C)$ be the augmentation ideal of $\Omega(\tilde C)$.
Then by Quillen \cite[\S1.3]{Quillen}
the cyclic complex of $C$, denoted by $(\mathrm{CC}_\bullet(C), b)$,
may take to be complex $(\bar{\Omega}(\tilde C)_{\natural}[-1],\partial)$.

\begin{proof}[Sketch of proof of Proposition \ref{Prop:cotangentcomplex}]
On the one hand, by \eqref{Koszulresolution}
we have 
\begin{equation}\label{quasiisoofcyclic1}
\mathrm{CC}_\bullet(\Omega(A^{\ac}))\simeq \mathrm{CC}_\bullet(A)
\end{equation}
since quasi-isomorphic DG algebras have quasi-isomorphic cyclic chain complexes.
On the other hand, by Jones-McCleary \cite[Theorem 1]{JM} 
(see also \cite[Lemma 17]{CYZ} for a proof from the Koszul duality point of view)
we have
\begin{equation}\label{quasiisoofcyclic2}
\mathrm{CC}_\bullet(\Omega(A^{\ac}))\simeq\mathrm{CC}_\bullet(A^{\ac}).
\end{equation}
Combining \eqref{quasiisoofcyclic1} and \eqref{quasiisoofcyclic2}
we obtain $\bar{\tilde R}_\natural\simeq \mathrm{CC}_\bullet(A)[1]$.

Next, we show
$\mathrm{DR}^1_{\mathrm{nc}}\tilde R\simeq\mathrm{CH}_\bullet(A)$.
Since 
$\Omega^1_{\mathrm{nc}}\tilde R[-1]=\tilde R\otimes A^{\ac}\otimes\tilde R$,
we have
$$
\mathrm{DR}^1_{\mathrm{nc}} \tilde R=
(\Omega^1_{\mathrm{nc}}\tilde R[-1])_{\natural}=A^{\ac}\otimes \tilde R=\mathrm{CH}_\bullet(A^{\ac}),
$$
where $\mathrm{CH}_\bullet(A^{\ac})$ is the underlying space of the Hochschild chain complex of $A^{\ac}$.
By a direct computation we also see that $\partial$ on $(\Omega^1_{\mathrm{nc}}\tilde R)_{\natural}$
coincides with the Hochschild boundary map on $\mathrm{CH}_\bullet(A^{\ac})$.

Again by Koszul duality, the same argument as above yields
$
\mathrm{CH}_\bullet(A)\simeq\mathrm{CH}_\bullet(A^{\ac})
$ (see \cite[Theorem 15]{CYZ} for a complete proof), 
which implies $\mathrm{DR}^1_{\mathrm{nc}} \tilde R\simeq\mathrm{CH}_\bullet(A)$.
\end{proof}

\begin{convention}\label{conv:cyclichomology}
In the rest of the paper, as adopted by Berest et. al. in \cite{BKR},
when writing $\mathrm{CC}_\bullet(-)$, we always mean
$\mathrm{CC}_\bullet(-)[1]$.
\end{convention}

\begin{proposition}\label{Prop:tangentcomplex}
Suppose $A$ is a Koszul algebra and $\tilde R=\Omega(\tilde A^{\ac})$.
Then
$((\mathbb D\mathrm{er}\;\tilde  R[1])_{\natural},\partial)$ is quasi-isomorphic to 
$(\mathrm{Der}\; \tilde R[1],\partial)$,
which is further quasi-isomorphic to the Hochschild cochain complex
$\mathrm{CH}^\bullet(A)$.
\end{proposition}

\begin{proof}
Observe that as graded vector spaces
\begin{eqnarray*}
\mathbb D\mathrm{er}\; \tilde R[1]
&=&\mathrm{Hom}_{\tilde R^e}(\Omega^1_{\mathrm{nc}}\tilde R[-1],\tilde R\otimes\tilde R)\\
&=&\mathrm{Hom}_{\tilde R^e}(\tilde R\otimes A^{\ac}\otimes\tilde R,\tilde R\otimes\tilde R)\\
&=&\mathrm{Hom}(A^{\ac},\tilde R\otimes\tilde R)\\
&=&\tilde R\otimes A^{!}\otimes \tilde R,
\end{eqnarray*}
thus 
$$
(\mathbb D\mathrm{er}\;\tilde R[1])_{\natural}=\mathrm{Der}\;\tilde R[1]
\cong \tilde R\otimes A^{!}
=\bigoplus_n\mathrm{Hom}((A^![-1])^{\otimes n}, A^!)
=\mathrm{CH}^\bullet(A^!).
$$
Under this identity, a direct calculation
identifies the differential on $\mathrm{Der}\;\tilde R$ 
with
the Hochschild coboundary on $\mathrm{CH}^\bullet(A^!)$.

Finally, by Keller \cite[Theorem 3.5]{Keller} (see also Shoikhet \cite[Theorem 4.2]{Shoikhet}), for a Koszul algebra $A$,
$\mathrm{CH}^\bullet(A)$ is quasi-isomorphic to $\mathrm{CH}^\bullet(A^!)$
as DG Lie algebras, and hence in particular, as chain complexes.
Thus combining it with the above quasi-isomorphisms, 
we have
$(\mathbb D\mathrm{er}\; \tilde R[1])_{\natural}\simeq \mathrm{CH}^\bullet(A)$.
\end{proof}

\section{Koszul Calabi-Yau algebras}\label{sect:KoszulCY}

The notion of Calabi-Yau algebras was introduced by Ginzburg \cite{Ginzburg} in 2007.
Let $A$ be a Koszul Calabi-Yau algebra and $\tilde R=\Omega(\tilde A^{\ac})$ be as before.
In this section, we show that the volume form of the noncommutative Poincar\'e duality
of $A$
also gives the shifted bi-symplectic structure on $\tilde R$.

\begin{definition}[Ginzburg]
Suppose $A$ is an associative algebra over $k$. Then $A$ is called {\it Calabi-Yau of
dimension $n$} (or {\it $n$-Calabi-Yau}) if
\begin{enumerate}
\item $A$ is homologically smooth, that is, $A$, viewed as a (left) $A^e$-module,
has a bounded, finitely generated projective resolution, and
\item there is an isomorphism
\begin{equation}\label{cond:CY(ii)}
\eta: \mathrm{RHom}_{A^e}(A, A\otimes A)\cong A[n]
\end{equation}
in the derived category $\mathbf D(A^e)$ of (left) $A^e$-modules.
\end{enumerate}
\end{definition}

In the above definition, 
the $A^e$-module structure on $A\otimes A$ and $\mathrm{RHom}_{A^e}(A, A\otimes A)$
is completely analogous to the case of $\mathbb D\mathrm{er}\; A$.

Suppose $A$ is a homologically smooth algebra, 
then Van den Bergh \cite{VdB0} showed that
\begin{eqnarray}\label{VdBsIso}
\mathrm{HH}_i(A) &=&\mathrm{H}_i(A\otimes^{\mathrm L}_{A^e} A)\nonumber\\
&\cong &
\mathrm{Hom}_{\mathbf D(A^e)}(\mathrm{RHom}_{A^e}(A, A\otimes A),A[i]),
\quad\mbox{for all}\; i
\end{eqnarray}
and therefore for $A$ being Calabi-Yau, the isomorphism \eqref{cond:CY(ii)}
corresponds to an element in $\mathrm{HH}_n(A)$, which is still denoted by $\eta$
and is called the {\it volume class} of $A$.

In general, for an arbitrary Calabi-Yau algebra $A$,
it is difficult to find its volume class.
However, in the case when $A$ is also Koszul, this turns out to be very easy.
Besides that, Koszul Calabi-Yau algebras have some other good features; for example,
they form so far the most known and interesting examples of Calabi-Yau algebras,
and they are all given by a {\it superpotential} (see \cite{VdB3}).

\subsection{The volume form}

Now suppose $A$ is Koszul Calabi-Yau of dimension $n$. 
First, by Proposition \ref{KoszulResolutionOfA}, we have
\begin{eqnarray*}
&&\mathrm{Hom}_{\mathbf D(A^e)}(\mathrm{RHom}_{A^e}(A, A\otimes A), A[n])\\
&=&
\mathrm{Hom}_{\mathbf D(A^e)}(\mathrm{Hom}_{A^e}(A\otimes A^{\ac}\otimes A, A\otimes A), 
A\otimes A^{\ac}[n]\otimes A)\\
&=&\mathrm{Hom}_{\mathbf D(A^e)}(\mathrm{Hom}_k(A^{\ac}, A\otimes A), 
A\otimes A^{\ac}[n]\otimes A)\\
&=&\mathrm{Hom}_{\mathbf D(A^e)}(A\otimes A^{!}\otimes A, A\otimes A^{\ac}[n]\otimes A).
\end{eqnarray*}
Second, 
by \eqref{cond:CY(ii)}, we get 
$$A\otimes A^{!}\otimes A\cong A\otimes A^{\ac}[n]\otimes A$$
in $\mathbf D(A^e)$. Since both sides of this identification
are minimal free resolutions of $A$, we get
an isomorphism of $A^{!}\cong A^{\ac}[n]$ as $A^!$-bimodules
(see also \cite[Proposition 28]{CYZ}).
This implies that $A^!$ is a {\it cyclic} associative
algebra of degree $n$.
Let us recall its definition first.

\begin{definition}[Cyclic associative algebra]
Suppose $A$ is a graded associative algebra.
It is called {\it cyclic} of degree $n$
if it admits a degree $n$, non-degenerate bilinear pairing
$\langle-,-\rangle: A\times A\to k[n]$
such that $$\langle a\cdot b,c\rangle=(-1)^{(|a|+|b|)|c|}\langle c\cdot  a,b\rangle,\quad\mbox{for all}\;
a,b,c\in A.$$
\end{definition}

\begin{proposition}[Van den Bergh]\label{dualitybetweenCYcyclic}
Suppose that $A$ is a Koszul algebra.
Then $A$ is $n$-Calabi-Yau
if and only if $A^{!}$ is cyclic of degree $n$.
\end{proposition}

\begin{proof}
See Van den Bergh \cite{VdB3}.
The interested reader may also refer to \cite[Proposition 28]{CYZ} for a simple proof.
\end{proof}

In the literature, a cyclic associative algebra is sometimes
also called a {\it symmetric Frobenius} algebra,
or simply a {\it symmetric} algebra.
If $A$ is a cyclic associative algebra of degree $n$,
then its linear dual $A^*=\mathrm{Hom}(A, k)$, which is a coassociative coalgebra,
also has a pairing
\begin{equation}
\langle-,-\rangle: A^*\times A^* \to  k[-n],\;
 (\alpha, \beta) \mapsto \langle\alpha^*, \beta^*\rangle,
\end{equation}
where $\alpha^*$ and $\beta^*$ are the images of $\alpha$ and $\beta$
under the map $A^{*}\stackrel{\cong}{\to} A$ induced by the pairing.
The cyclicity condition of the pairing becomes
\begin{equation}\label{cyclicconditionforcoalg}
\sum_{(\beta)}\langle \alpha, \beta^1\rangle\cdot  \beta^2
=\sum_{(\alpha)}(-1)^{|\alpha||\beta^1|}\langle \beta, \alpha^2\rangle \cdot \alpha^1,
\end{equation}
for any $\alpha,\beta\in A^*$ with
$\Delta(\alpha)=\sum_{(\alpha)}\alpha^1\otimes \alpha^2$ and 
$\Delta(\beta)=\sum_{(\beta)}\beta^1\otimes \beta^2$.

The above proposition in fact implies that the volume class of a Koszul Calabi-Yau algebra
is represented by a(ny) nonzero top-degree element in $A^{\ac}$.
More precisely, we have the following.

\begin{proposition}
Suppose $A$ is a Koszul $n$-Calabi-Yau algebra.
Then any nonzero element
$\eta\in A^{\ac}_{n}$, viewed
as an element in
$A\otimes A^{\ac}$ via the embedding
$A^{\ac}\cong k\otimes A^{\ac}\subset A\otimes A^{\ac}$
and hence a chain
in $(A\otimes A^{\ac}_{\bullet}, b)$,
is a cycle and represents
the volume class of $A$.
\end{proposition}

\begin{proof}
First from the symmetric pairing
$$
\langle x, y\rangle=(-1)^{|x||y|}\langle y,x\rangle$$
we see that in particular
$$
\xi_i\eta=(-1)^{n}\eta \xi_i,\quad\mbox{for all}\; \xi_i\in \Sigma^{-1}W^*.
$$
This means $\eta$ via the above embedding
$A^{\ac}\cong k\otimes A^{\ac}\subset A\otimes A^{\ac}$ is a cycle, namely,
$$b(1\otimes \eta)=\sum_i 
e_i\otimes \xi_i\eta-(-1)^n e_i\otimes \eta \xi_i=0.
$$
Also from the following operation
\begin{equation}\label{capgivesPD}
\begin{array}{ccl}
(A\otimes A^{!}, b)&\longrightarrow& (A\otimes A^{\ac}, b)\\
a\otimes f&\longmapsto& a\otimes (\eta\cap f)
\end{array}
\end{equation}
it is direct to check that this is an isomorphism
of chain complexes and induces an isomorphism
on the homology
$$
\mathrm{H}^\bullet(A\otimes A^!, b)\cong\mathrm{H}_{n-\bullet}(A\otimes A^{\ac}, b).
$$
Thus by Proposition \ref{lemma:identityofHochschild}
we have an isomorphism
\begin{equation}\label{NCPDofVdB}
\mathrm{HH}^\bullet(A)\cong\mathrm{HH}_{n-\bullet}(A).
\end{equation}
We next need to check that the volume class represented above is identical
to the one of Van den Bergh (\cite{VdB0}).

First, in \cite{dTdVVdB}
de Thanhofer de Volcsey and Van den Bergh showed that
the noncommutative Poincar\'e duality for Calabi-Yau algebras
is given by a class $\eta\in\mathrm{HH}_n(A)$
such that
\begin{equation}
\begin{array}{ccl}
\mathrm{HH}^\bullet(A)&\longrightarrow&\mathrm{HH}_{n-\bullet}(A)\\
f&\longmapsto&\eta\cap f
\end{array}
\end{equation}
is an isomorphism,
where the ``cap product" $\cap$ is given on the Hochschild chain level as follows:
\begin{equation}
\begin{array}{ccl}
\bar{\mathrm{CH}}_{m}(A)\times \bar{\mathrm{CH}}^{n}(A)&\stackrel{\cap}
\longrightarrow &\bar{\mathrm{CH}}_{m-n}(A)\\
(\alpha,\quad\quad f )&\longmapsto&
\left\{
\begin{array}{cl}
(a_0f(\bar{a}_1,\ldots,\bar{a}_n),\bar{a}_{n+1},\ldots,\bar{a}_m),&\mbox{if}\; m\geq n\\
0,&\mbox{otherwise,}
\end{array}\right.
\end{array}
\end{equation} 
where $\bar{\mathrm{CH}}_\bullet(A)$ and $\bar{\mathrm{CH}}^\bullet(A)$
are the reduced Hochschild chain and cochain complexes respectively
({\it c.f.} Loday \cite{Loday} for these notions).

Second, on the Koszul complexes we have an analogous cap product
given by the following (note that \eqref{capgivesPD} is just a special case)
\begin{equation}
\begin{array}{ccl}
(A\otimes A^{\ac})\times (A\otimes A^!)&\stackrel{\cap}\longrightarrow& A\otimes A^{\ac}\\
(a\otimes u,  b\otimes f)&\longmapsto& ab\otimes ( uf).
\end{array}
\end{equation}
It has been shown by Berger et. al. \cite{BLS} that for Koszul algebras,
these two versions of cap product on the homology level, via the isomorphism given
in Proposition \ref{lemma:identityofHochschild}
are the same.
Thus the isomorphism \eqref{NCPDofVdB}
is a version of noncommutative Poincar\'e duality in the sense of Van den Bergh.

Third, by Van den Bergh \cite{VdB3} the volume class of the noncommutative
Poincar\'e duality, if it exists, is unique up to an inner automorphism of $A$.

Thus by the above three arguments, the noncommutative Poincar\'e duality 
of \eqref{NCPDofVdB} is identical to the one of Van den Bergh given in \cite{VdB0},
possibly up to an inner automorphism of $A$.
This completes the proof.
\end{proof}

Now consider the coproduct
$$\Delta(\eta)=\sum  \eta^1_i\otimes\eta^2_i\in A^{\ac}\otimes A^{\ac}.$$
Observe that we have an embedding
\begin{equation}\label{embeddingofAac}
A^{\ac} \cong k\otimes A^{\ac} \otimes k\subset \tilde R\otimes
A^{\ac} \otimes \tilde R=\Omega^1_{\mathrm{nc}}\tilde R[-1],\quad a\mapsto d([a]).
\end{equation}
Via this embedding,
$\Delta(\eta)$ corresponds to
an element
$$\omega:=\sum d([\eta_i^1]) \otimes d([\eta_i^2])
\in\Omega^{2}_{\mathrm{nc}}\tilde R.$$

\begin{lemma}\label{volumeformasdRcycle}
$\omega$
descends to a $\partial$- and $d$-closed
cycle of degree $n$ in $\mathrm{DR}_{\mathrm{nc}}^2 \tilde R$.
\end{lemma}

\begin{proof}
Denote by $[\omega]$
the image of $\omega$ in $\mathrm{DR}^{2}_{\mathrm{nc}}\tilde R$.
We show $[\omega]$ is closed with respect to both $d$ and $\partial$. First, $[\omega]$ is automatically
$d$-closed. 
Second, applying $\partial$ to $\omega$, we have
\begin{eqnarray*}
\partial(\omega)
&=&\sum \partial\circ d ([\eta_i^1]) \otimes d ([\eta_i^2])+(-1)^{|\eta_i^1|}d([\eta_i^1])\otimes
\partial\circ d([\eta_i^2]) \\
&=&-\sum d\circ\partial ([\eta_i^1]) \otimes d ([\eta_i^2])+(-1)^{|\eta_i^1|}d([\eta_i^1])\otimes
d\circ\partial ([\eta_i^2]) \\
&=&-\sum (-1)^{|\eta_i^{11}|}d( [\eta_i^{11}|\eta_i^{12}]) \otimes d ([\eta_i^2])
+(-1)^{|\eta_{i}^1|+|\eta_i^{21}|}d([\eta_i^1])\otimes
d([\eta_i^{21}|\eta_i^{11}])\\
&=&-\sum (-1)^{|\eta_i^{11}|} d([\eta_i^{11}])\cdot [\eta_i^{12}]\otimes d([\eta_i^2])
-[\eta_i^{11}]\cdot d([\eta_i^{12}]) \otimes d([\eta_i^2])\\
&&-\sum(-1)^{|\eta_{i}^1|+|\eta_i^{21}|} 
d([\eta_i^1])\otimes d([\eta_i^{21}])\cdot [\eta_i^{22}]-
 (-1)^{|\eta_{i}^1|}d([\eta_i^1])\otimes [\eta_i^{21}]\cdot
d([\eta_i^{22}]),
\end{eqnarray*}
where we write $\Delta (\eta_i^1)=\sum \eta_i^{11}\otimes \eta_i^{12}$
and $\Delta (\eta_i^2)=\sum \eta_i^{21}\otimes \eta_i^{22}$.
In the last equality, after descending to $\mathrm{DR}^2_{\mathrm{nc}}\tilde R$,
the first and the last summands cancel with each other
due to the co-associativity of $A^{\ac}$, while the second and the third summands
cancel with each other
due to the cyclic condition of the pairing on $A^!$, which is equivalent to \eqref{cyclicconditionforcoalg}. 
This proves the statement.
\end{proof}

\begin{theorem}\label{thm:mainthm1}
Let $A$ be a Koszul Calabi-Yau algebra of dimension $n$.
Let $\tilde R=\Omega(\tilde A^{\ac})$. Then
$\tilde R$
has a $(2-n)$-shifted bi-symplectic structure.
\end{theorem}

\begin{proof}
Since $\Omega^1_{\mathrm{nc}} \tilde R[-1]=\tilde R\otimes  A^{\ac}\otimes \tilde R$,
we have
\begin{eqnarray*}
\mathbb D\mathrm{er}\; \tilde R[1]&=&
\mathrm{Hom}_{\tilde R^e}(\Omega^1_{\mathrm{nc}} \tilde R[-1], \tilde R\otimes \tilde R)\\
&=&\mathrm{Hom}_{\tilde R^e}(\tilde R\otimes A^{\ac}\otimes \tilde R, \tilde R\otimes \tilde R)\\
&=&\mathrm{Hom}( A^{\ac}, \tilde R\otimes \tilde R)\\
&=&\mathrm{Hom}(A^{\ac}, k)\otimes(\tilde R\otimes\tilde R).
\end{eqnarray*}
For any $f\otimes r_1\otimes r_2\in\mathrm{Hom}(A^{\ac}, k)\otimes(\tilde R\otimes\tilde R)$,
by \eqref{formula:reducedcontraction}
its reduced contraction with $\omega$
is 
$$
2\cdot\sum  f(\eta_i^1)\cdot r_2\otimes d\eta_i^2\otimes r_1
\in\Omega^1_{\mathrm{nc}}\tilde R[-1].
$$
In other words,
the reduced contraction is essentially given
by 
$
A^!\stackrel{\eta\cap(-)}\longrightarrow A^{\ac}[n]
$.
Since $\eta$ and thus
$\omega$ are non-degenerate of total degree $n$, we thus have
\begin{equation}\label{iso:tangentandcotangent}
\mathbb D\mathrm{er}\tilde R[1]\cong(\Omega^1_{\mathrm{nc}}\tilde R[-1])[2-n]
\end{equation}
as $\tilde R$-bimodules.
This proves the theorem.
\end{proof}

By taking the commutator quotient space of both sides of the 
above \eqref{iso:tangentandcotangent} and 
Propositions \ref{Prop:cotangentcomplex} and \ref{Prop:tangentcomplex},
we once 
again obtain the noncommutative Poincar\'e duality:
$$
\mathrm{HH}^\bullet(A)\cong\mathrm{HH}_{n-\bullet}(A),
$$
which coincides with Van den Bergh's one.
In general, suppose $R$ has a (shifted) bi-symplectic structure $\omega$,
then Crawley-Boevey et. al. showed in \cite[Lemma 2.8.6]{CBEG} that
there is a commutative diagram
$$
\xymatrixcolsep{4pc}
\xymatrix{
\mathbb D\mathrm{er}\; R\ar[d]_{\iota_{(-)}\omega}^{\cong}\ar[r]^{\natural}&
(\mathbb D\mathrm{er}\; R)_{\natural}
\ar[d]^{(\iota_{(-)}\omega)_{\natural}}_{\cong}\\
\Omega^1_{\mathrm{nc}}R\ar[r]^{\natural}&\mathrm{DR}^1_{\mathrm{nc}}R.
}
$$

\section{Representation schemes and the shifted symplectic structure}\label{sect:defofsss}

In this section, we briefly go over the relationship between the shifted bi-symplectic
structure of a DG algebra and the shifted symplectic structure on its DG representations.

\subsection{Representation functors}
Let $\mathbf{DGA}$ be the category of associative, unital DG $k$-algebras,
and $\mathbf{CDGA}$ its subcategory of DG commutative $k$-algebras.
Fix a finite dimensional vector space $V$, and consider the following functor:
\begin{equation*}
\mathrm{Rep}_V(A): \mathbf{CDGA}\to \mathbf{Sets},\;
B \mapsto \mathrm{Hom}_{\mathbf{DGA}}(A, B\otimes\mathrm{End}\; V).
\end{equation*}
The following result generalizes the result of Bergman \cite{Bergman} and Cohn \cite{Cohen} 
for associative algebras:

\begin{theorem}[\cite{BKR} Theorem 2.2]\label{idoftwosets}
The functor $\mathrm{Rep}_V(A)$ is representable, that is, 
there exists an object in $\mathbf{CDGA}$, say $A_V$, depending only on 
$A$ and $V$, 
such that
\begin{equation}\label{equiv:homsets}
\mathrm{Hom}_{\mathbf{DGA}}(A, B\otimes\mathrm{End}\; V)=
\mathrm{Hom}_{\mathbf{CDGA}}(A_V, B).
\end{equation}
\end{theorem}

More precisely, in \cite{BKR} the authors considered the following two functors:
\begin{equation*}
\sqrt[V]{-}:\mathbf{DGA}\to\mathbf{DGA},\; A\mapsto(A\ast_k\mathrm{End}\; V)^{\mathrm{End}\; V}.
\end{equation*}
and
\begin{equation*}
(-)_{\natural\natural}:\mathbf{DGA}\to\mathbf{CDGA},\; A\mapsto A/\langle [A,A]\rangle,
\end{equation*}
where $\langle [A,A]\rangle$ is the ideal of $A$ generated by the commutators.
Then the functor
\begin{equation}\label{functor:Rep} 
(-)_V:\mathbf{DGA}\to\mathbf{CDGA},\; A\mapsto A_V
\end{equation}
is given by the composition of the above two functors, namely,
$
A_V=(\sqrt[V]{A})_{\natural\natural}$.
In what follows we sometimes 
also write $A_V$ as $\mathrm{Rep}_V(A)$.

In \eqref{equiv:homsets},
if we take $B=A_V$, then we have
\begin{equation}\label{equiv:universalmap}
\mathrm{Hom}(A, A_V\otimes\mathrm{End}\; V)=\mathrm{Hom}(A_V, A_V).
\end{equation}
The identity map on the right
hand side
corresponds to a map
$$\pi_V: A\to A_V\otimes\mathrm{End}\; V$$
on the left hand side,
which is usually called the {\it universal representation map}.

\begin{example}[Rep for quasi-free algebras]

For a quasi-free algebra $(R, \partial)$ and $V=k^n$, the DG commutative algebra
$R_V$ can be described explicitly:
Let
$\{x^\alpha\}_{\alpha\in I}$ be a set of generators of $R$.
Consider a free graded algebra
$R'$ on generators
$\{x_{ij}^\alpha: 1\le i, j\le n, \alpha\in I\}$,
where
$|x_{ij}^\alpha|=|x^\alpha|$ for all $i, j$. 
Form matrices
$X^{\alpha}:=(x_{ij}^{\alpha})$, and
define
the algebra
map
$$\pi: R\to M_n(R'),\quad x^\alpha\mapsto X^{\alpha},$$
where $M_n(R')$ is the algebra of $n\times n$-matrices with entries in $R'$.
Let
$$\partial(x_{ij}^{\alpha}):=(\pi(\partial x^{\alpha}))_{ij},$$
and extend it to $R'$ by linearity and the Leibniz rule.
We thus obtain a DG algebra $(R', \partial)$,
and 
$R_V$ is $R_{\natural\natural}'$ with the differential
induced from $\partial$ (see \cite[Theorem 2.8]{BKR} for a proof).
\end{example}

\subsection{$\mathrm{GL}$-invariants and the trace map}

Observe that $\mathrm{GL}(V)$ acts on $\mathrm{Rep}_V(A)$ by conjugation.
More precisely, 
then
any $g\in\mathrm{GL}(V)$
gives a unique automorphism of $A_V$ 
which, under the identity \eqref{equiv:homsets}, corresponds to the composition
$$A\stackrel{\pi_V}\longrightarrow \mathrm{End}(A_V)
\stackrel{\mathrm{Ad}(g)}\longrightarrow \mathrm{End}(A_V).$$
This action is natural in $A$, and hence defines a functor
\begin{equation}\label{functor:GLinvariants}
\mathrm{Rep}_V(-)^{\mathrm{GL}}: \mathbf{DGA}\to\mathbf{CDGA},\; A\mapsto (A_V)^{\mathrm{GL}},
\end{equation}
where $(-)^{\mathrm{GL}}$ means the $\mathrm{GL}(V)$-invariants.

Now consider the following composite map
$$
A\stackrel{\pi_V}{\longrightarrow} \mathrm{End}(A_V)
\stackrel{\mathrm{Tr}}\longrightarrow A_V,$$
which is $\mathrm{GL}(V)$-invariant
and factors through
$A_{\natural}$,
we get a map
\begin{equation}\label{tracemap}
\mathrm{Tr}: A_\natural\to(A_V)^{\mathrm{GL}}.
\end{equation}
If $A$ is an associative algebra
(i.e., a DG algebra concentrated in degree 0),
then the famous result of Procesi says
that the image of $\mathrm{Tr}$
generates
$\mathrm{Rep}_V(A)^{\mathrm{GL}}$; in other words,
if we extend $\mathrm{Tr}$ to be a commutative algebra map
$$\mathrm{Tr}: \mathbf{\Lambda}^\bullet A_\natural\to 
(A_V)^{\mathrm{GL}},$$ 
then it is surjective.
However, for an arbitrary DG algebra, Berest
and Ramadoss showed in \cite{BR} that
this is in general not true on the homology level; there are some homological
obstructions for $\mathrm{Tr}$ to be so.


\subsection{Van den Bergh's functor}
Let $(R,\partial)$ be a DG algebra.
Suppose $(M,\partial_M)$ is a DG $R$-bimodule.
Let $\pi: R\to R_V\otimes\mathrm{End}\; V$ be the universal
representation of $R$,
which means the map on the left hand side of \eqref{equiv:homsets} 
that corresponds to the identity map on the right hand side.
Then $\pi$ gives an $R$-bimodule
structure on $R_V\otimes\mathrm{End}\; V$.
Denote
\begin{equation}\label{functor:VdB}
M_V:=M\otimes_{R^e} (R_V\otimes\mathrm{End}\; V),
\end{equation}
which is now a DG $R_V$-module.
More specifically, let $V=k^n$, then
$M_V$ is generated by symbols $m_{ij}$, $1\le i, j\le n$, for each $m\in M$, with the action of $R_V$
given by
$$ (r\cdot m)_{ij}=\sum_{k}r_{ik}\cdot m_{kj}, \quad
(m\cdot r)_{ij}=\sum_{k} r_{kj}\cdot m_{ik},$$
and with the differential, denoted by $\partial_{M_V}$, given by
$$
\partial_{M_V}(m_{ij})=(\partial_M m)_{ij},\quad\mbox{for all}\; m\in M.
$$
The assignment from the category of DG $R$-bimodules
to the category of DG $R_V$-modules
$$
\mathbf{DGBimod}\; R\to\mathbf{DGMod}\; R_V,\; M\mapsto M_V
$$
is a well-defined functor, and is first introduced by Van den Bergh in \cite{VdB2}.
Next, we apply Van den Berg's idea to the case of noncommutative
differential forms and poly-vectors.

Let $(R, \partial)$ be a DG commutative algebra over $k$. Let
$$
I:=\mathrm{ker}(R\otimes R\stackrel{\mu }
\longrightarrow R)
$$
be the kernel of the multiplication map
and let $\Omega^1_{\mathrm{com}}R:=I/I^2$, which is the set of {\it K\"ahler differentials} of $R$.
Let 
$$\Omega^p_{\mathrm{com}}R=\mathbf{\Lambda}_R^p(\Omega^1_{\mathrm{com}}R[-1]).$$
Similarly to the DG algebra case, we have the degree $1$ de Rham differential
$$
d: \Omega^\bullet_{\mathrm{com}}R\to \Omega^{\bullet}_{\mathrm{com}}R,
$$
which makes $(\Omega^\bullet_{\mathrm{com}}R, d)$ into a DG cochain algebra.
The differential $\partial$ on $R$ also gives a degree $-1$ differential on $\Omega^\bullet_{\mathrm{com}}R$,
which also respects the product and commutes with $d$. 

The dual space of the cotangent space
$
\mathrm{Hom}_R(\Omega^1_{\mathrm{com}} R, R)
$
is called {\it the complex of vector fields} of $R$, and is identified with $\mathrm{Der}\; R$.

\begin{proposition}[\cite{VdB2} Proposition 3.3.4]\label{prop:VdBonformsandfields}
Suppose $R$ is a DG algebra.
Then \begin{enumerate}
\item $\big(\Omega^1_{\mathrm{nc}}R\big)_V=\Omega^1_{\mathrm{com}}(R_V);$
\item $\big(\mathbb D\mathrm{er}\; R\big)_V=\mathrm{Der}\; R_V$.
\end{enumerate}
\end{proposition}

From this proposition, we immediately have: 

\begin{proposition}[\cite{VdB2} Corollary 3.3.5]
Suppose $R$ is a quasi-free DG algebra. Then
\begin{equation}\label{identityofformsandvectors}
\big(T_R(\Omega^1_{\mathrm{nc}}R[-1])\big)_V=\Omega^\bullet_{\mathrm{com}}(R_V)\quad
\mbox{and}\quad
\big(T_R(\mathbb D\mathrm{er}\;R[1])\big)_V=\mathbf\Lambda^\bullet (\mathrm{Der}\;R_V[1]).
\end{equation}
\end{proposition}

Now applying \eqref{tracemap} to
\eqref{identityofformsandvectors},
we have the trace maps
\begin{equation}\label{tracemapofquotientspaces}
(T_R(\Omega^1_{\mathrm{nc}}R[-1]))_{\natural}=\mathrm{DR}^\bullet_{\mathrm{nc}}R\to
\Omega^\bullet_{\mathrm{com}}(R_V)^{\mathrm{GL}}
\end{equation}
and
$$
(T_R(\mathbb D\mathrm{er}\;R[1]))_{\natural}\to
\mathbf\Lambda^\bullet (\mathrm{Der}\;R_V[1])^{\mathrm{GL}}.
$$
When restricting to the first component, we have
the trace map
\begin{equation}\label{tracemapofquotientspaces1}
\mathrm{DR}^1_{\mathrm{nc}}R
\to \Omega^1_{\mathrm{com}}(R_V)^{\mathrm{GL}}
\quad\mbox{and}\quad
(\mathbb D\mathrm{er}\; R[1])_{\natural}\to
(\mathrm{Der}\;R_V[1])^{\mathrm{GL}}.
\end{equation}





\subsection{The shifted symplectic structure}

The notion of shifted symplectic structure is introduced
by Pantev-To\"en-Vaqui\'e-Vezzosi in \cite{PTVV}; see also
\cite{CPTVV,Melani,Pridham} for some further studies.
In the following we only consider the affine case, which is enough for our purpose.

Suppose $(R,\partial)$ is a DG algebra. For any closed form $\omega\in \Omega^2_{\mathrm{com}}R$ of 
total degree $2-n$, 
the {\it contraction map}
$$
\iota_{(-)}\omega: \mathrm{Der}\; R [1]\to (\Omega^1_{\mathrm{com}} R[-1])[2-n],\;
\alpha\mapsto\omega(\alpha, -)[2-n]
$$
is a map of $\partial$-complexes.
The following is a slightly stronger version of the shifted symplectic structure
introduced in \cite{PTVV}.

\begin{definition}[Shifted symplectic structure]
Suppose $(R,\partial)$ is a DG commutative algebra over $k$.
An {\it $n$-shifted symplectic structure} on $R$ is a
2-form $\omega\in\Omega^2_{\mathrm{com}}R$
of total degree $2-n$, closed under $\partial$ and $d$,
such that the contraction
\begin{equation}\label{mapfromnctangent}
\iota_{(-)}\omega: \mathrm{Der}\;R [1]\to(\Omega_{\mathrm{com}}^1 R[-1])[2-n],
\end{equation}
is a quasi-isomorphism.
\end{definition}

\begin{remark}\label{rmk:degreeshifting}
The original definition of shifted symplectic structure in
\cite{PTVV} requires $\omega$ to be $\partial$-closed, which can
be extended to be a closed form in the negative cyclic complex associated to the
mixed complex $\Omega_{\mathrm{com}}^\bullet(R)$.

In both \eqref{isofrombisymplectic} and \eqref{mapfromnctangent}
we have shifted the degrees on the right hand side of the equations, namely on
the (noncommutative)
differential 1-forms, up by $n-2$, which looks different from \cite{PTVV}.
However, they are the same in the following sense:
\eqref{isofrombisymplectic} and \eqref{mapfromnctangent}
can be alternatively written as
$$
\mathbb D\mathrm{er}\; R\to\Omega^1_{\mathrm{nc}} R[-n]\quad
\mbox{and}\quad
\mathrm{Der}\; R\to\Omega^1_{\mathrm{com}} R[-n]
$$
respectively, which coincides with \cite{PTVV}.
Such a degree shifting guarantees that an $n$-shifted symplectic structure
gives an $n$-shifted Poisson structure, whose Poisson bracket has degree $n$.
\end{remark}

\begin{theorem}[\cite{CBEG}; see also \cite{VdB2} \S2.4]
Suppose $R$ is a DG algebra which admits an $n$-shifted bi-symplectic structure.
Then $R_V$ has an $n$-shifted symplectic structure.
\end{theorem}
\begin{proof}
Follows from Proposition \ref{prop:VdBonformsandfields} and the functoriality of Van den Bergh's functor.
More precisely, applying Van den Bergh's functor to 
$$
\xymatrix{
\mathbb D\mathrm{er}\; R[1]\ar[r]^-{\iota_{(-)}\omega}_-{\cong}&
(\Omega^1_{\mathrm{nc}} R[-1])[n]
}
$$
and then using Proposition  \ref{prop:VdBonformsandfields}
we obtain 
\begin{equation}\label{diag:bisymplectictosymplectic}
\xymatrix{
\mathrm{Der}(R_V)[1]\ar[r]^-{\cong}&
(\Omega^1_{\mathrm{com}}(R_V)[-1])[n],
}
\end{equation}
where the isomorphism, according to \cite[Theorem 6.4.3]{CBEG},
is given by ${\iota_{(-)}\mathrm{Tr}(\omega)}$.
\end{proof}

Combining the above theorem with Theorem \ref{thm:mainthm1},
we immediately have the following:

\begin{corollary}
Suppose $A$ is a Koszul Calabi-Yau algebra of dimension $n$,
and let $\tilde R=\Omega(\tilde A^{\ac})$ as before. Then
$\tilde R_V$ has an $(2-n)$-shifted symplectic structure.
\end{corollary}




\subsection{Identification of $\mathrm{GL}$-invariant 1-forms and vectors}

In this subsection, we show that $\mathrm{Rep}_V(\tilde R)^{\mathrm{GL}}$ is 
``symplectic".  What we mean is the following:

\begin{theorem}\label{conj:Lam}
Let $A$ be a Koszul Calabi-Yau algebra of dimension $n$.
Then
$$
\mathrm{Der}(\tilde R_V)^{\mathrm{GL}}[1]
\cong
(\Omega^1_{\mathrm{com}}(\tilde R_V)^{\mathrm{GL}}[-1])[n].
$$
\end{theorem}

\begin{proof}
Observe that in \eqref{diag:bisymplectictosymplectic},
$\mathrm{Tr}(\omega)$ in fact lies in $\Omega^2_{\mathrm{com}}(\tilde R_V)^{\mathrm{GL}}$ (see \eqref{tracemapofquotientspaces}).
Thus by taking the $\mathrm{GL}(V)$-invariant vector fields and 1-forms, 
we get the desired isomorphism.
\end{proof}

Combining this proposition with Propositions \ref{Prop:cotangentcomplex}
and \ref{Prop:tangentcomplex},
we obtain that the trace map \eqref{tracemapofquotientspaces1}
gives the following commutative diagram of chain complexes
\begin{equation}\label{diag:tracefromHochschildcomplex}
\xymatrixcolsep{4pc}
\xymatrix{
\mathrm{CH}^\bullet(A)\ar[r]^-{\mathrm{Tr}}\ar[d]^{\cong}
&\mathrm{Der}(\tilde R_V)^{\mathrm{GL}}[1]\ar[d]^{\cong}\\
\mathrm{CH}_{n-\bullet}(A)\ar[r]^-{\mathrm{Tr}}
&(\Omega^1_{\mathrm{com}}(\tilde R_V)^{\mathrm{GL}}[-1])[n].
}
\end{equation}
Taking the (co)homology on both sides, we get the 
the commutative diagram \eqref{digram:PDandshiftedsymplectic}
stated in \S\ref{sect:intro}.
As we remarked before, the vertical maps are both isomorphisms, while
the two horizontal maps are neither surjective nor injective
in general
(for see \cite{BR} more details).





\section{The shifted double Poisson structure}\label{sect:shiftedPoisson}

Shifted bi-symplectic structures
are intimately related to shifted double Poisson structures.
Let us remind the work \cite{VdB} of Van den Bergh (here we
rephrase it in the DG case; see also \cite{BCER}).

\begin{definition}[Double bracket]
Suppose $R$ is a DG algebra over $k$.
A \textit{double bracket} of degree $n$ on $R$ is a DG map
$\ldb-,-\rdb: R\times R\to R\otimes R$
of degree $n$ which is a derivation in its second argument and
satisfies
\begin{eqnarray}
 \ldb a,b\rdb&= &-(-1)^{(|a|+n)(|b|+n)}\ldb b,a\rdb^\circ, 
 \end{eqnarray}
where $(u\otimes v)^\circ =(-1)^{|u||v|}v\otimes u$.
\end{definition}

\begin{definition}[Double Poisson structure]\label{Def:DPs}
Suppose that $\ldb-,-\rdb$ is a double bracket of degree $n$ on $R$.
For $a,b_1,...,b_n\in R$, let
$$\ldb a, b_1 \otimes \cdots \otimes b_n\rdb_L \,:=\, \ldb a,b_1 \rdb \otimes b_2 \otimes \cdots \otimes b_n,$$
and for $s$ is a permutation of $\{1,2,\cdots, n\}$, let
$$\sigma_s(b_1 \otimes \cdots \otimes b_n):=(-1)^{\sigma(s)}b_{s^{-1}(1)} \otimes \cdots \otimes b_{s^{-1}(n)},$$
where $$\sigma(s)=\displaystyle\sum_{i<j; s^{-1}(j)<s^{-1}(i)}|a_{s^{-1}(i)}||a_{s^{-1}(j)}|.$$
If furthermore $R$ satisfies the following
\textit{double Jacobi identity}
\begin{multline}\label{dJ}
\ldbg a , \ldb b,c \rdb \rdbg_L +(-1)^{(|a|+n)(|b|+|c|)} \sigma_{(123)}\ldbg b,\ldb c,a\rdb \rdbg_L\\
 + (-1)^{(|c|+n)(|a|+|b|)}\sigma_{(132)}
\ldbg c,\ldb a,b\rdb \rdbg_L =0,
\end{multline}
then $R$ is called a {\it double Poisson algebra}
of degree $n$ (or \textit{$n$-shifted double Poisson algebra}).
\end{definition}



Van den Bergh \cite{VdB}  showed that,
if $R$ is equipped with a double Poisson structure,
then there is a Poisson structure on the affine scheme $\mathrm{Rep}_V(R)$
of all the representations of 
$R$ in $V$.
%
Independently and simultaneously, Crawley-Boevey gave in \cite{CB} the {\it weakest} condition for 
$\mathrm{Rep}_V(R)/\!/{\mathrm{GL}(V)}$ of $R$ to have a Poisson structure. 
He called such condition the {\it $\mathrm H_0$-Poisson structure}, since
it involves the
zeroth Hochschild/cyclic homology of $R$.
It turns out Van den Bergh's
double Poisson structure satisfies this condition, and is so far the most interesting example
therein. 


\subsection{From shifted bi-symplectic to shifted double Poisson}

In \cite[Appendix]{VdB} Van den Bergh showed that a bi-symplectic
structure gives a double Poisson structure.
We rephrase it in the Koszul Calabi-Yau case.
Suppose $A$ is a Koszul $n$-Calabi-Yau algebra,
and let $\tilde R=\Omega(\tilde A^{\ac})$.
For any
$r\in\tilde R$,
since
$$
\iota_{(-)}\omega: \mathbb D\mathrm{er}\;\tilde R[1]\to(\Omega^1_{\mathrm{nc}}\tilde R[-1])[n]
$$
is an {\it isomorphism} of chain complexes,  
there exists an element $H_r\in\mathbb D\mathrm{er}\;\tilde R$
such that
$$
\iota_{H_r}\omega=d(r).
$$
$H_r$ is called the {\it bi-Hamiltonian vector field} associated to $r$.
Now 
consider the following bracket
$$
\ldb-,-\rdb :  \tilde R\times \tilde R \to  \tilde R\otimes\tilde R,\;
 (r_1,r_2) \mapsto H_{r_1}(r_2),
$$
then we have:

\begin{proposition}\label{prop:shiftedbPontildeR}
$\ldb-,-\rdb $ gives a $(2-n)$-shifted double Poisson structure on $\tilde R$.
\end{proposition}

\begin{proof}
By Van den Bergh \cite[Lemma A.3.3]{VdB},
the $(2-n)$-shifted double Poisson bracket 
on $\tilde R=\Omega(\tilde A^{\ac})$ is given by the following
formula:
\begin{equation}\label{def:doublePoissonbracket}
\ldb x, y\rdb:=\sum_{i=1}^k\sum_{j=1}^\ell
(-1)^{\sigma_{ij}}
(x_i, y_j)\cdot
( y_1\cdots y_{j-1}  x_{i+1}\cdots  x_k)
\otimes
( x_1\cdots  x_{i-1} y_{j+1}\cdots  y_{\ell}),
\end{equation}
for $x=( x_1\cdots  x_k)$ and $y=( y_1\cdots  y_{\ell})$ in $\tilde R$,
where $(-1)^{\sigma_{ij}}$ is the Koszul sign.
Here $(x_i, y_j)\in k$ for $x_i, y_j\in A^{\ac}[1]$ is the 
graded skew-symmetric pairing induced from the pairing on $A^{\ac}$;
more precisely, $(x_i, y_j)=(-1)^{n-|x_i|}\langle\Sigma x_i, \Sigma y_j\rangle$.

Formula \eqref{def:doublePoissonbracket}
also appeared in \cite[Theorem 15]{BCER}
(see also \cite[Lemma 4.4]{CEEY}) for $R=\Omega(A^{\ac})$,
where we also showed that
$\ldb-,-\rdb$ given above
commutes with the differential on $R$.
The only difference between $\tilde R$ and $R$ is that
the generators of $\tilde R$ contain one more element, namely
the desuspension of the co-unit of $A^{\ac}$.
The sufficient condition for $\ldb-,-\rdb$ commuting with the differentials on $R$ and on $\tilde R$
is the cyclic condition \eqref{cyclicconditionforcoalg} for $A^{\ac}$.
Thus the proof of \cite[Theorem 15]{BCER}
 applies to the above proposition, too.
\end{proof}




Suppose $R$ has a double Poisson structure,
then Van den Bergh gave an explicit formula
for the Poisson structure on $\mathrm{Rep}_V(R)$ (see \cite[Propositions 7.5.1 and 7.5.2]{VdB}).
For a Koszul Calabi-Yau
algebra $A$, the Poisson structure on $\mathrm{Rep}_n(\tilde R)$ is given as follows:
suppose $A^{\ac}$ has a set of basis $\{x^\alpha\}_{\alpha\in I}$, and
$\Delta(x^{\alpha})=\sum_{(x^{\alpha})} x^{\alpha_1}\otimes x^{\alpha_2}$.
Then $\mathrm{Rep}_n(\tilde R)$ is the quasi-free DG commutative algebra generated by
\begin{equation}\label{RepofcobarofKoszulCY}
\big\{x_{ij}^{\alpha}\;\big|\;\alpha\in I, 1\le i,j\le n, |x_{ij}^{\alpha}|=|x^{\alpha}|-1\big\},
\end{equation}
with
$$
d(x_{ij}^{\alpha})=\sum_{(x^{\alpha})}(-1)^{|x^{\alpha_1}|}\sum_{k=1}^{n}x_{ik}^{\alpha_1}\cdot x_{kj}^{\alpha_2}.
$$
The $(2-n)$-Poisson bracket is given by
\begin{equation}\label{PoissononDRep}
\{x_{ij}^{\alpha}, x_{k\ell}^\beta\}=(-1)^{n-|x^{\alpha}|}\delta_{i\ell}\delta_{jk}\langle x^\alpha, x^\beta\rangle
\end{equation}
on the generators, which extends to the whole $\mathrm{Rep}_n(\tilde R)$ by the Leibniz rule.

\subsection{The work of Crawley-Boevey}

In \cite{CB} Crawley-Boevey introduced what he called the {\it $\mathrm{H}_0$-Poisson structure}.
Let us recall its definition:

\begin{definition}[$\mathrm H_0$-Poisson structure]
Suppose $R$ is an associative algebra. An {\it $\mathrm H_0$-Poisson structure}
on $R$ is a Lie bracket
$$
\{-,-\}: R_{\natural}\times R_{\natural}\to R_{\natural}
$$
such that the adjoint action
$$
ad_{\bar u}: R_\natural\to R_\natural, \; \bar v\mapsto \{\bar u,\bar v\}
$$
can be lifted to be a derivation
$$
d_{u}: R\to R,
\quad\mbox{for all}\; u\in R.$$
\end{definition}

Crawley-Boevey proved that if $R$ admits an $\mathrm H_0$-Poisson structure, 
then there is a unique Poisson structure on $\mathrm{Rep}_V(R)^{\mathrm{GL}}$
such that the trace map
$$
\mathrm{Tr}: R_{\natural}\to \mathrm{Rep}_V(R)^{\mathrm{GL}}
$$
is a map of Lie algebras, or in other words,
$
\mathrm{Tr}: \mathbf\Lambda^\bullet R_{\natural}\to \mathrm{Rep}_V(R)^{\mathrm{GL}}
$
is a map of Poisson algebras. Here the Poisson bracket on $\mathbf\Lambda^\bullet R_{\natural}$
is the extension of the bracket on $R_{\natural}$ by derivation.

Now suppose $R$ has a double Poisson bracket $\ldb-,-\rdb$.
Then
$$\{-,-\}: R\times R\to R,\; (u,v)\mapsto\mu \circ \ldb u,v\rdb$$
descends to a well-defined Lie bracket (see \cite[Lemma 2.4.1]{VdB})
\begin{equation}\label{fromdbtoLie}
\{-,-\}: R_\natural\times R_\natural\to R_{\natural},
\end{equation}
which is an $\mathrm H_0$-Poisson structure.
The restriction of the Poisson structure on $\mathrm{Rep}_V(R)$
gives the one on $\mathrm{Rep}_V(R)^{\mathrm{GL}}$.

This result was later generalized to the DG setting in \cite{BCER}.
For the Koszul Calabi-Yau algebra case, we have the following.

\begin{corollary}
Suppose that $A$ is a Koszul Calabi-Yau algebra, and let $\tilde R=\Omega(\tilde A^{\ac})$ be as above.
The trace map
\begin{equation}\label{diag:tracefromcycliccomplex}
\mathrm{Tr}: \tilde R_{\natural}\to\mathrm{Rep}_V(\tilde R)^{\mathrm{GL}}
\end{equation}
is a map of degree $(2-n)$ DG Lie algebras.
As a consequence,
\begin{equation}\label{tracemapofPoissonalgebras}
\mathrm{Tr}: \mathbf\Lambda^\bullet\tilde R_{\natural}\to
\mathrm{Rep}_V(\tilde R)^{\mathrm{GL}}\end{equation}
is a map of $(2-n)$-shifted Poisson algebras.
\end{corollary}

\begin{proof}
See \cite[Corollary 3]{BCER}.
Again we emphasize that in \cite{BCER} we proved
the statement for $R=\Omega(A^{\ac})$, but the same proof
applies to the current case.
\end{proof}

Alternatively, one can prove the above corollary by 
considering the $\mathrm{GL}(V)$-invariant elements
in $\mathrm{Rep}_V(\tilde R)$ (see \cite[Theorem 3.1]{BR})
and then applying \eqref{PoissononDRep} to them directly.
We leave it to the interested readers.

\subsection{Relations with quivers and quiver representations}\label{subsect:relationswithquivers}

The shifted bi-symplectic and double Poisson structures on $\tilde R$
generalize the ones of quivers given in \cite{CBEG,VdB}.

Let $Q$ be a quiver. For simplicity let us assume
$Q$ has only one vertex. Let $\bar Q$ be the double of $Q$.
Then the path algebra $k\bar Q$
is an associative algebra over $k$;
viewing $\bar Q$ as a 1-dimensional CW complex, 
then $k\bar Q$ is exactly the cobar construction of
the coalgebra of the chain complex of $\bar Q$.

Denote the set of the edges of $Q$ by $\{e_i\}$ and their duals
by $\{e_i^*\}$,
then there is a graded symmetric pairing on $\{e_i\}\cup\{e_i^*\}$ given by
$$
\langle e_i, e_j\rangle=\langle e_i^*,e_j^*\rangle=0,\quad
\langle e_i, e_j^*\rangle=-\langle e_j^*, e_i\rangle=\delta_{ij}
$$
Such pairing is non-degenerate and cyclically invariant, and therefore
by results in previous sections,
the cobar construction of $\bar Q$, that is, $k\bar Q$, has
a $0$-shifted bi-symplectic and double Poisson structure.
The bi-symplectic and double Poisson structures are exactly the ones
obtained in \cite{CBEG,VdB}. The corresponding symplectic
structure on $\mathrm{Rep}_V(k\bar Q)$ as well as the Lie bracket
on $k\bar Q_{\natural}$ was also previously studied by Ginzburg in \cite{Ginzburg01};
see also Bocklandt-Le Bruyn \cite{BLB}, where the Lie algebra on $k\bar Q_{\natural}$ is called
the {\it necklace Lie algebra}.

Now assign the gradings of the edges $\bar Q$ other than one and obtain a DG coalgebra, 
say $C$. 
To obtain the shifted bi-symplectic and double Poisson structures
on $\Omega(C)$ (respectively $\Omega(\tilde C)$, where $\tilde C$
is the co-unitalization of $C$), 
then a sufficient condition is that $\bar C=C\backslash k$ (respectively $C$)
is cyclic, in other words, the dual space $\mathrm{Hom}(\bar C, k)$ (respectively
$\mathrm{Hom}(C, k)$) is a cyclic associative and not necessarily unital algebra.

\section{Quantization}\label{sect:quantization}

In this section, we study the quantization problem. 
In \cite{Schedler}, Schedler proved that the necklace Lie algebra
of a doubled quiver is in fact an involutive Lie bialgebra,
and constructed a Hopf algebra which quantizes this Lie bialgebra.
He also showed that the Hopf algebra is mapped to the
Moyal-Weyl quantization of the quiver representation spaces
as associative algebras.

Later in \cite{GS}, he together with Ginzburg
constructed a Moyal-Weyl type quantization of the necklace Lie bialgebra, 
and showed that
such quantization is isomorphic to the Hopf algebra constructed in \cite{Schedler}.

The purpose of this section is to generalize their results
to the Koszul Calabi-Yau case. 
The main result is the commutative diagram \eqref{Diag:quantization}.
Some partial results have been previously
obtained in \cite{CEG}.

\subsection{Quantization of $\mathrm{Rep}_n(\tilde R)$}

In the bivector form, the shifted
Poisson structure on $\mathrm{Rep}_n(\tilde R)$ is given by
\begin{equation}\label{formula:Poissonstructure}
\pi=\sum_{\alpha,\beta\in I}\sum_{i,j}
(-1)^{n-|x^{\alpha}|}
\langle x^{\alpha},x^{\beta}\rangle\frac{\partial}{\partial x^{\alpha}_{ij}}\wedge
\frac{\partial}{\partial x^{\beta}_{ji}}.
\end{equation}
Observe that $\pi$ is of constant coefficients,
and thus we have the {\it Moyal-Weyl
quantization} of $\mathrm{Rep}_n(\tilde R)$ 
which is
given by
\begin{equation}\label{def:starprodonDRep}
f\star g:=\mu \circ e^{\frac{\hbar}{2}\pi}(f\otimes g)\in \mathrm{Rep}_n(\tilde R)[\hbar],\quad
\mbox{for any}\; f, g\in \mathrm{Rep}_n(\tilde R),
\end{equation}
where ``$\mu$" is the original multiplication on $\mathrm{Rep}_n(\tilde R)$, and $\hbar$ is a formal
parameter of degree $n-2$.

\begin{proposition}
$(\mathrm{Rep}_n(\tilde R)[\hbar],\star)$ is a DG associative algebra over $k[\hbar]$, which quantizes
$\mathrm{Rep}_n(\tilde R)$.
\end{proposition}

\begin{proof}
We only need to show that the differential $d$ commutes with $\star$,
or equivalently,
$\partial$ commutes with $\mu \circ  \pi^r(-,-)$, for all $r\in\mathbb N$.

In fact, since $\pi^r(f, g)$ is of $r$-th order for both arguments, we only need to check
the case when $f$ and $g$ are degree $r$ monomials.
In this case, $\mu \circ \pi^n(f, g)$ is a number, whose differential is zero, and hence
we need to check
\begin{equation}\label{boundarycommuteswithquantization}
\mu \circ\pi^r(\partial f, g)+\mu \circ\pi^r(f, \partial g)=0.
\end{equation}
Suppose $f=x_{i_1j_1}^{\alpha_1}\cdots x_{i_rj_r}^{\alpha_r}, 
g=y_{k_1\ell_1}^{\beta_1}\cdots y_{k_r\ell_r}^{\beta_r}$.
Then up to sign, 
$\mu \circ\pi^r(\partial f, g)$ and $\mu \circ\pi^r(f, \partial g)$
both contain a common scalar factor which is obtained by
applying $\pi^{r-1}$ to $x_{i_1j_1}^{\alpha_1}\cdots\widehat
{x_{i_pj_p}^{\alpha_p}}\cdots x_{i_rj_r}^{\alpha_r}$
and 
$y_{k_1\ell_1}^{\beta_1}\cdots \widehat{y_{k_q\ell_q}^{\beta_q}}\cdots y_{k_r\ell_r}^{\beta_r}$
where $\widehat{\;\;\;\;}$ means the corresponding component is omitted.
The rest factors are
just $\pi(d({x_{i_pj_p}^{\alpha_p}}), y_{k_q\ell_q}^{\beta_q})$
and $\pi({x_{i_pj_p}^{\alpha_p}}, d(y_{k_q\ell_q}^{\beta_q}))$ respectively.
This means, to prove \eqref{boundarycommuteswithquantization}
it is sufficient to show
$
\pi(f, g)
$
commutes with the boundary, which is already done.
This proves the statement.
\end{proof}

\subsection{Quantization of $\tilde R_\natural$}

In this subsection we study the quantization of the Lie bialgebra on $\tilde R_{\natural}$.
Let us start with several definitions.

\begin{definition}[Lie bialgebra]
Suppose that  $(L,\{-,-\})$ is a graded Lie algebra with the bracket $\{-,-\}$ having degree $m$
and $(L,\delta)$ is a graded Lie coalgebra with the cobracket $\delta$ having degree $n$.
The triple $(L,\{-,-\},\delta)$ is called a  {\it Lie bialgebra} of degree $(m,n)$ 
if the following Drinfeld compatibility
(also called the cocycle condition) holds:
for all $a, b\in L$,
\begin{equation}\label{Drinfeld com}
  \delta(\{a,b\})=(ad_a\otimes id+id\otimes ad_a)\delta (b)+(ad_b\otimes id+id\otimes ad_b)\delta (a),
\end{equation}
where $ad_a(b)=\{a,b\}$ is the adjoint action.
If furthermore, $\{-,-\}\circ\delta(g)\equiv 0$ for any $g\in L$, then the Lie bialgebra is called {\it involutive}.
\end{definition}

In the above definition, if $L$ is equipped with a differential which commutes with both
the Lie bracket and Lie cobracket, then it is called a {\it DG Lie bialgebra}.

Now let $A$ be a Koszul Calabi-Yau algebra.
Recall that the DG Lie bracket on $\tilde R_{\natural}$
is given by the following formula ({\it c.f.} \eqref{def:doublePoissonbracket} and \eqref{fromdbtoLie}):
\begin{equation}
\{ x, y\}:=\sum_{i=1}^k\sum_{j=1}^\ell
(-1)^{\sigma_{ij}}
( x_i, y_j)\cdot\mathrm{pr}
( y_1\cdots y_{j-1}  x_{i+1}\cdots  x_kx_1\cdots  x_{i-1} y_{j+1}\cdots  y_{\ell}),
\end{equation}
for $x, y\in\tilde R_{\natural}$
represented by $( x_1\cdots  x_k)$ and $y=( y_1\cdots  y_{\ell})$ in $\tilde R$,
where $\mathrm{pr}(-)$ means the projection of $\tilde R$ to $\tilde R_{\natural}$.
Now define
$$\delta:\tilde R_{\natural}\longrightarrow  \tilde R_{\natural}\otimes
\tilde R_{\natural}$$
by
\begin{multline*}
\delta( x_1 x_2 \cdots  x_n)\\
:=\sum_{i,j:\, i<j} (-1)^{\sigma_{ij}}(x_i, x_j)\cdot\mathrm{pr}
( x_1 \cdots x_{i-1} x_{j+1}  \cdots  x_n)\otimes\mathrm{pr} ( x_{i+1} \cdots  x_{j-1})\\
-\sum_{i,j:\, i<j} (-1)^{\sigma_{ij}'}(x_i, x_j)\cdot\mathrm{pr} ( x_{i+1} \cdots  x_{j-1})
\otimes \mathrm{pr}( x_1 \cdots  x_{i-1} x_{j+1} \cdots  x_n).
\end{multline*}

\begin{theorem}[\cite{CEG} Theorem 1(i)]
$(\tilde R_{\natural}, \{-,-\},\delta)$
forms an involutive DG Lie bialgebra of degree $(2-d, 2-d)$.
\end{theorem}


This Theorem, together with the following Theorem \ref{thm:existenceofquantization},
is proved in \cite{CEG}.
They are directly inspired by \cite{Schedler}; what is new there is that
the Lie bracket and cobracket thus defined are compatible with the differential.
As we remarked in the proof of Proposition \ref{prop:shiftedbPontildeR}, 
the cyclic condition \eqref{cyclicconditionforcoalg} 
guarantees that all constructions respect it. 
We thus omit the proof and refer the interested 
reader to \cite{CEG} for more details.


\begin{definition}[Quantization of Lie bialgebra]
Suppose $(L, \{-,-\},\delta)$ is a DG Lie bialgebra of degree $(m, n)$ over the field $k$.
A {\it quantization} of $L$ is a Hopf algebra $(A,\star,\Delta)$, flat over $k[\hbar,h]$, where $\hbar$
and $h$
are formal parameters of degree $-m$ and $-n$ respectively, together with a surjective map
$$
\Psi: A\to \mathbf\Lambda^\bullet L
$$
such that for all $x, y\in A$,
\begin{equation}\label{def:quantization}
\Psi\Big(\frac{x\star y-y\star x}{\hbar}\Big)=\{\Psi(x),\Psi(y)\},\quad
\Psi\Big(\frac{\Delta(x)-\Delta^{\mathrm{op}}(x)}{h}\Big)=\delta(\Psi(x)),
\end{equation}
where
$\Delta^{\mathrm{op}}(-)$ means the opposite coproduct,
and $\delta: \mathbf\Lambda^\bullet L\to  \mathbf\Lambda^\bullet L\otimes \mathbf\Lambda^\bullet L$ is the
cobracket induced by $\delta: L\to L\otimes L$.
\end{definition}

Let us remind that 
the cobracket on $ \mathbf\Lambda^\bullet L$ is defined as follows:
for any Lie coalgebra $(L,\delta)$, its graded symmetric product $ \mathbf\Lambda^\bullet L$ admits
a co-Poisson algebra structure which is induced by the Lie coalgebra structure $\delta$ 
via the formula:
$$\delta\circ \mu=(\mu\otimes \mu)\circ(1\otimes\tau\otimes 1)\circ (\delta\otimes\Delta+\Delta\otimes\delta),$$
where $\mu$ is the product of $ \mathbf\Lambda^\bullet L$, 
$\tau:a\otimes b\mapsto(-1)^{|a||b|} b\otimes a$ is the switching
operator, and $\Delta$ is the coproduct of $ \mathbf\Lambda^\bullet L$:
$$
\Delta(x^I)=\sum_{I_1\cup I_2=I}x^{I_1}\otimes x^{I_2}.
$$

In the rest of this subsection, we study the quantization of $\tilde R_{\natural}$.
Recall that $\tilde R=T(A^{\ac}[1])$ and $A^{\ac}$ is a cyclic coalgebra.

Now let
$LA^{\ac}=A^{\ac}\otimes k[\nu,\nu^{-1}]$,
where $\nu$ is a formal parameter of degree $0$.
An element $a\otimes \nu^r$ in $LA^{\ac}$ is denoted by $(a, r)$.
Let $T(LA^{\ac}[1])$ be the free tensor algebra of the desuspension of $LA^{\ac}$,
where its elements are written in the form 
$
[(a_{1}, r_{1})|\cdots|(a_{p}, r_{p})]
$, for $(a_1, r_1),\cdots, (a_p, r_p)\in LA^{\ac}$.
Let $LH$ 
be the commutator quotient space of $T(LA^{\ac}[1])$, namely
$LH=T(LA^{\ac}[1])_{\natural}$.
To avoid complicated notations, 
in the following we write elements in 
$LH$ again in the form
$
[(a_{1}, r_{1})|\cdots|(a_{p}, r_{p})]
$; in other words we omit the symbol $\mathrm{pr}(-)$.

Let $SLH$ be the graded symmetric algebra generated by $LH$,
whose product is denoted by $\bullet$.
Define a differential $\partial$ on $SLH$ given by the following formula:
\begin{multline}\label{dofSLH}
\partial\big([(a_{1,1}, r_{1,1})|\cdots|(a_{1, p_1}, r_{1, p_1})]
\bullet\cdots\bullet[(a_{s,1}, r_{s,1})|\cdots|(a_{s, p_s}, r_{s, p_s})]\big)\\
:=\sum_{i=1}^s\sum_{j=1}^{p_s}\sum_{(a_{i,j})}(-1)^{\sigma_{ij}}
[(a_{1,1},\tilde r_{1,1})|\cdots|(a_{1, p_1}, \tilde r_{1, p_1})]\bullet\cdots\quad\quad\quad\quad\\
\bullet[(a_{i,1}, \tilde r_{i,1})|\cdots|
(a_{i,j}',r_{i,j})|(a_{i,j}'', 1+r_{i,j})|\cdots
(a_{i, p_i}, \tilde r_{i, p_i})]\bullet\cdots,
\end{multline}
where $(-1)^{\sigma_{ij}}$ is the Koszul sign, $a_{i,j}'$ and $a_{i,j}''$ come from
$\Delta(a_{i,j})=\sum_{(a_{i,j})}a_{i,j}'\otimes a_{i,j}''$,
and
for $(i',j')\ne (i,j)$,
$$
\tilde r_{i',j'}=\left\{
\begin{array}{ll}
r_{i',j'},&\mbox{if}\; r_{i',j'}\le r_{i,j}\\
1+r_{i',j'},&\mbox{if}\; r_{i',j'}>r_{i,j}.
\end{array}
\right.
$$
The coassociativity of $A^{\ac}$ implies that $b^2=0$.

Now let $\hbar, h$ be formal parameters both of degree $n-2$,
and let
$SLH[\hbar,h]$ be the $k[\hbar,h]$-module genenrated by $SLH$,
on which $\partial$ extends $k[\hbar,h]$-linearly.
Let $\widetilde{SLH}$ be the subcomplex of $SLH[\hbar,h]$ spanned by 
\begin{equation}\label{elementsinH}
[(a_{1,1}, r_{1,1})|\cdots|(a_{1, p_1}, r_{1, p_1})]
\bullet\cdots\bullet[(a_{s,1}, r_{s,1})|\cdots|(a_{s, p_s}, r_{s, p_s})]
\end{equation}
where $r_{i,j}$ are all distinct.

Consider the quotient space of $\widetilde{SLH}$
by identifying 
$$[(a_{1,1}, r_{1,1})|\cdots|(a_{1, p_1}, r_{1, p_1})]
\bullet\cdots\bullet[(a_{s,1}, r_{s,1})|\cdots|(a_{s, p_s}, r_{s, p_s})]$$
with
$$
[(a_{1,1}, r_{1,1}')|\cdots|(a_{1, p_1}, r_{1, p_1}')]
\bullet\cdots\bullet[(a_{s,1}, r_{s,1}')|\cdots|(a_{s, p_s}, r_{s, p_s}')]
$$
under the condition that
$r_{i,j}<r_{i',j'}$ if and only if
$r_{i,j}'<r_{i',j'}'$.
Denote this quotient space by $\tilde A$.
Pick an element in $\tilde A$, suppose it is represented by
$[(a_{1,1}, r_{1,1})|\cdots|(a_{1, p_1}, r_{1, p_1})]
\bullet\cdots\bullet[(a_{s,1}, r_{s,1})|\cdots|(a_{s, p_s}, r_{s, p_s})]$,
without loss of generality, we may assume all $r_{i,j}$ are even,
then the image of \eqref{dofSLH}
represents an element in $\tilde A$.
This means that $\tilde A$ with the differential induced by $b$ is a chain complex.

Now let $\tilde B$ be the submodule of $\tilde A$ generated by elements of the following form:
\begin{enumerate}
\item[(1)]
$X- X_{i,j,i',j'}'-\hbar \cdot X_{i,j,i',j'}''$, 
where $i\ne i'$, $r_{i,j}<r_{i',j'}$, and there does not exist
$(i'',j'')$ with $r_{i,j}<r_{i'',j''}<r_{i,j'}$, and

\item[(2)] $X- X_{i,j,i,j'}'- h\cdot  X_{i,j,i,j'}''$, 
where $r_{i,j}<r_{i,j'}$, and there does not exist
$(i'',j'')$ with $r_{i,j}<r_{i'',j''}<r_{i,j'}$,
\end{enumerate}
where $X_{i,j,i',j'}'$ and $X_{i,j,i',j'}''$ are given as follows:
if $i\ne i'$, then $X_{i,j,i',j'}'$ is the same as $X$ except that $r_{i,j}$ and $r_{i',j'}$ are interchanged,
while
$X_{i,j,i',j'}''$ replaces
the factors
$[(a_{i,1}, r_{i,1})|\cdots|(a_{i,p_i}, r_{i, p_i})]$
and 
$[(a_{i',1}, r_{i',1})|\cdots|(a_{i',p_{i'}}, r_{i', p_{i'}})]$
by
$$(-1)^{\sigma_{iji'j'}}(a_{i,j}, a_{i',j'})
[(a_{i,j+1}, r_{i,j+1})|\cdots|(a_{i,j-1}, r_{i,j-1})|(a_{i',j'+1}, r_{i',j'+1})|\cdots|(a_{i',j'-1}, r_{i',j'-1})];
$$
similarly,
$X_{i,j,i,j'}'$ is the same as $X$ but with $r_{i,j}$ and $r_{i,j'}$ interchanged, while
$X_{i,j,i,j'}''$
replaces the factor with the following factor:
$$
(-1)^{\sigma_{ijij'}}( a_{i,j}, a_{i,j'})
[(a_{i,j'+1}, r_{i,j'+1})|\cdots|(a_{i,j-1}, r_{i,j-1})]\bullet[(a_{i,j+1}, r_{i,j+1})|\cdots|(a_{i,j'-1}, r_{i,j'-1})].
$$
It is proved in \cite[Lemma 14]{CEG} that $\tilde B$ is a subcomplex of $\tilde A$.
Let $H=\tilde A/\tilde B$. 

\begin{theorem}[\cite{CEG} Theorem 15]\label{thm:existenceofquantization}
There
is a DG Hopf algebra structure on $H$ over $k[\hbar,h]$, which quantizes
$\tilde R_{\natural}$.
\end{theorem}

The product on $H$, denoted by $\star$, is easy to describe (as we will only use it):
for two elements in $H$, say $X$ and $Y$, suppose they are both
represented by elements in the form
\eqref{elementsinH},
 raise those $r_{i,j}$ in $Y$
such that they are all greater than those in $X$, then
the product of $X$ and $Y$, $X\star Y$, 
is represented by $X\bullet Y$.




\subsection{Lifting the trace map}
In this subsection we relate the quantization of $\tilde R_{\natural}$ with the one
of $\mathrm{Rep}_V(\tilde R)$. We set $\hbar=h$.

Extending the trace map \eqref{tracemapofPoissonalgebras}
by $k[\hbar]$-linearity, we obtain a $k[\hbar]$-linear map
$$
\widetilde{\mathrm{Tr}}: \mathbf\Lambda^\bullet\tilde R_{\natural}[\hbar] \to
 \mathrm{Rep}_V(\tilde R)[\hbar].
$$
Clearly $\widetilde{\mathrm{Tr}}$ commutes with the differential.
We have the following.

\begin{theorem}\label{thm:traceofquantization}
For a Koszul Calabi-Yau algebra $A$, the 
following map
\begin{equation}\label{tracemapofquantizedPoisson}
\widetilde{\mathrm{Tr}}:(\mathbf\Lambda^\bullet\tilde R_{\natural}[\hbar], \star )\to
(\mathrm{Rep}_V(\tilde R)[\hbar], \star).
\end{equation}
is a map of DG algebras over $k[\hbar]$.
\end{theorem}

\begin{proof}
Since $\widetilde{\mathrm{Tr}}$ commutes with the differential on both sides,
we only need to show it is a graded algebra map; however, this is already done in Schedler 
\cite[\S3.4]{Schedler}. 

More precisely, for
a quiver $Q$, if we denote its double by $\bar Q$
then what Schedler constructed in \cite{Schedler} is the
following: 
\begin{enumerate}
\item[(1)] a Lie bialgebra structure, called the necklace Lie bialgebra,
on the commutator quotient space $(k\bar Q)_\natural$;
\item[(2)] a Hopf algebra $H$ over $k[\hbar]$, completely analogous to the construction in
previous subsection,
quantizing the necklace Lie bialgebra $(k\bar Q)_{\natural}$; 
\item[(3)] an algebra map from
this Hopf algebra to the Weyl algebra of differential operators
on $\mathrm{Rep}_V(kQ)$.
\end{enumerate}
Note that $\mathrm{Rep}_V(k\bar Q)$ is the cotangent space
of $\mathrm{Rep}_V(k Q)$, the algebra
of differential operators is exactly the Moyal-Weyl quantization
of $\mathrm{Rep}_V(k\bar Q)$.

As we remarked in \S\ref{subsect:relationswithquivers},
if we assume the the number of vertices of $Q$ is one, and the edges
of $Q$ are graded, then $\tilde R$ in the current paper is exactly the path algebra $k\bar Q$.
Therefore Schedler's construction and proof hold in our case.
\end{proof}

Later Ginzburg and Schedler showed in \cite{GS} that the quantization
constructed above is also of Moyal-Weyl type.  

In summary, we obtain the following commutative diagram
\begin{equation}\label{Diag:quantization}
\xymatrixcolsep{4pc}
\xymatrix{
(\mathbf\Lambda^\bullet\tilde R_{\natural}[\hbar], \star )
\ar[r]^{\widetilde{\mathrm{Tr}}}\ar@{~>}[d]_{\textup{quantization}}
&(\mathrm{Rep}_V(\tilde R)[\hbar], \star)\ar@{~>}[d]^{\textup{quantization}}\\
\mathbf\Lambda^\bullet\tilde R_{\natural}
\ar[r]^{\mathrm{Tr}}
&
\mathrm{Rep}_V(\tilde R).
}
\end{equation}
Recall that the images of $\mathrm{Tr}$ in fact lie in
$\mathrm{Rep}_V(\tilde R)^{\mathrm{GL}}$,
and the restriction of the Poisson structure
on $\mathrm{Rep}_V(\tilde R)$
gives the Poisson structure on 
$\mathrm{Rep}_V(\tilde R)^{\mathrm{GL}}$.
From \eqref{Diag:quantization}
we thus obtain the following commutative diagram
\begin{equation}\label{Diag:quantizationviatrace}
\xymatrixcolsep{4pc}
\xymatrix{
(\mathbf\Lambda^\bullet\tilde R_{\natural}[\hbar], \star )
\ar[r]^{\widetilde{\mathrm{Tr}}}\ar@{~>}[d]_{\textup{quantization}}
&(\mathrm{Rep}_V(\tilde R)^{\mathrm{GL}}[\hbar], \star)\ar@{~>}[d]^{\textup{quantization}}\\
\mathbf\Lambda^\bullet\tilde R_{\natural}
\ar[r]^{\mathrm{Tr}}
&
\mathrm{Rep}_V(\tilde R)^{\mathrm{GL}}.
}
\end{equation}
Recall that $\bar{\tilde R}_{\natural}\simeq\mathrm{CC}_\bullet(A)$
(see Proposition \ref{Prop:cotangentcomplex} and Convention \ref{conv:cyclichomology})
and observe that the bracket of any element in $\tilde R_\natural$
with unit of $\tilde R$ vanishes, we have an embedding
$\mathrm{HC}_\bullet(A)\hookrightarrow\mathrm{H}_\bullet(\tilde R_{\natural})$ as graded
vector spaces,
where $\mathrm{HC}_\bullet(A)$ is the cyclic homology of $A$. 
By pulling the Lie bialgebra structure on $\mathrm{H}_\bullet(\tilde R_{\natural})$,
$\mathrm{HC}_\bullet(A)$ thus has a Lie bialgebra structure of degree $(2-n, 2-n)$. 
Thus taking the homology on both sides of \eqref{Diag:quantizationviatrace}
and combining the above argument we in fact
obtain a commutative diagram
\begin{equation*}
\xymatrixcolsep{4pc}
\xymatrix{
(\mathbf\Lambda^\bullet\mathrm{HC}_\bullet(A)[\hbar], \star )
\ar[r]^{\widetilde{\mathrm{Tr}}}\ar@{~>}[d]_{\textup{quantization}}
&\big(\mathrm H_\bullet(\mathrm{Rep}_V(\tilde R))^{\mathrm{GL}}[\hbar], \star\big)\ar@{~>}[d]^{\textup{quantization}}\\
\mathbf\Lambda^\bullet\mathrm{HC}_\bullet(A)
\ar[r]^{\mathrm{Tr}}
&
\mathrm H_\bullet(\mathrm{Rep}_V(\tilde R))^{\mathrm{GL}},
}
\end{equation*}
where on the right hand side, we have used the fact that
$\mathrm{H}_\bullet(\mathrm{Rep}_V(\tilde R)^{\mathrm{GL}})\cong
\mathrm H_\bullet(\mathrm{Rep}_V(\tilde R))^{\mathrm{GL}}$ (see \cite[Theorem 2.6(b)]{BKR} 
for a proof).



\section{Derived representation schemes}\label{sect:DRep}

In this section, we briefly discuss the results in previous sections with
the derived representation schemes, introduced
by Berest, Khachatryan and Ramadoss.
The interested reader may refer to \cite{BCER,BFR,BKR,BR} for more details.

\subsection{Derived representation schemes}


In algebraic geometry,
there is an equivalence
of categories between affine
schemes and commutative algebras,
and many geometric structures on
affine schemes
have an algebraic description,
and vice versa. 
However, for associative algebras,
this correspondence does not exist.
In 1998, Kontsevich and Rosenberg \cite{KR} proposed
a heuristic principle to study
the {\it non-commutative geometry} on
associative, not-necessarily commutative, algebras, which is roughly stated as follows:
for an associative algebra over a field $k$, say $A$, any non-commutative geometric structure,
such as non-commutative Poisson, non-commutative symplectic, etc., 
should induce its classical counterpart in a natural way on
its representation scheme $\mathrm{Rep}_V(A)$, for all $k$-vector space $V$.

During the past decade, much progress has been made in the study of non-commutative
geometry under the guidance of the Kontsevich-Rosenberg Principle (see \cite{CBEG,Ginzburg01,VdB2}). 
However, for a general algebra $A$, $\mathrm{Rep}_V(A)$ is very singular.
In 2011 Berest et. al. \cite{BKR} suggested
that 
one should instead consider the {\it derived} representation schemes of $A$.
In this case, one replaces $A$ with its cofibrant resolution, say $QA$, in the category
of differential graded algebras,
and then considers the DG representation schemes of $QA$,
which are then smooth in the DG sense.
The cofibrant resolution of a DG algebra is not unique, but unique up to
homotopy.
Correspondingly, $\mathrm{Rep}_V(QA)$ is also unique up to homotopy.
By modulo such ambiguity, the DG representation scheme of $QA$ in $V$
in the homotopy category of DG commutative 
algebras,
denoted by $\mathrm{DRep}_V(A)$ and called the {\it derived representation scheme} of $A$ in $V$,
is a very good object that we can successfully apply the Kontsevich-Rosenberg principle.
Formally, their result is stated as follows.

\begin{theorem}[\cite{BKR}]
For any two objects $A, B\in\mathbf{DGA}$ and any $f\in\mathrm{Hom}(A, B)$, 
let $QA$ and $QB$ be any cofibrant replacement of $A$ and $B$ respectively,
and $Qf\in\mathrm{Hom}(QA,QB)$ be the corresponding cofibrant lifting of $f$. 
The functor \eqref{functor:Rep} has a total left derived functor
$$
\mathbf L(-)_V: \mathbf{Ho}(\mathbf{DGA})\to\mathbf{Ho}(\mathbf{CDGA}),\quad A\mapsto (QA)_V,\; f\mapsto (Qf)_V.
$$
\end{theorem}

According to \cite{BKR},
$\mathbf L(A)_V$ is called the {\it derived representation scheme}
(or {\it DRep} for short)
of $A$ in $V$, and is also denoted by $\mathrm{DRep}_V(A)$; its homology
$\mathrm{H}_\bullet(\mathrm{DRep}_V(A))$ is called
the {\it representation homology} of $A$, and is sometimes denoted by $\mathrm H_\bullet(A, V)$.


\begin{example}[DReps of Koszul algebras]
Suppose $A$ is a Koszul algebra. Then $A$ has an explicit cofibrant resolution 
$\Omega(A^{\ac})\stackrel{\sim}{\twoheadrightarrow} A$.
Thus $\mathrm{DRep}_n(A)$ is explicit in the Koszul case:
suppose $\bar A^{\ac}$ has a set of basis $\{x^\alpha\}_{\alpha\in I}$, and
the reduced coproduct $\bar\Delta(x^{\alpha})=\sum_{(x^{\alpha})} x^{\alpha_1}\otimes x^{\alpha_2}$.
Then $\mathrm{DRep}_n(A)$ is the quasi-free DG commutative algebra generated by
\begin{equation}\label{DRepofKoszulCY}
\big\{x_{ij}^{\alpha}\;\big|\;\alpha\in I, 1\le i,j\le n, |x_{ij}^{\alpha}|=|x^{\alpha}|-1\big\},
\end{equation}
with
$$
\partial(x_{ij}^{\alpha})=\sum_{(x^{\alpha})}(-1)^{|x^{\alpha_1}|}\sum_{k=1}^{n}x_{ik}^{\alpha_1}\cdot x_{kj}^{\alpha_2}.
$$
\end{example}

\subsection{Comparison of $\mathrm{DRep}_n(A)$ and $\mathrm{Rep}_n(\tilde R)$}

Suppose $A$ is a Kosuzl algebra.
Then one obtains $R=\Omega A^{\ac}$
from
$\tilde R=\Omega\tilde A^{\ac}$
by identifying the counit of $A^{\ac}$ with zero. 
Thus correspondingly,
if we identify the coordinates of $\mathrm{Rep}_V(\tilde R)$
corresponding the co-unit of $A^{\ac}$ with zero, we obtain
$\mathrm{Rep}_V(R)$, which is isomorphic to
$\mathrm{DRep}_V(A)$ in the homotopy category of DG algebras.

We here give two remarks regarding the role of the co-unit of $A^{\ac}$.
On the one hand, 
since $R=\Omega(A^{\ac})\simeq A$ is a cofibrant resolution of $A$,
$\mathrm{Rep}_V(R)$
is a derived functor of $A$.
 
On the other hand, considering
$\mathrm{Rep}_V(\tilde R)$ has the advantage when considering the $\mathrm{GL}$-invariants;
namely the $\mathrm{GL}(V)$-invariant functions, the cotangent vector space
and the tangent vector space on $\mathrm{Rep}_V(\tilde R)$ 
is directly related
to the cyclic complex, the Hochschild chain and cochain complexes of $A$ via the trace
maps (see \eqref{diag:tracefromHochschildcomplex} and \eqref{diag:tracefromcycliccomplex}). 
Sometimes, keeping the unit (respectively co-unit) of an algebra
(respectively coalgebra) is important
in the study of Noncommutative Geometry; see for example the early works
of Connes and Quillen in this field ({\it c.f.} \cite{Connes,Quillen}).

Now suppose $A$ is also Calabi-Yau.
In \cite{BCER,CEEY}, we showed that $R$ has a shifted Poisson structure, while
in the current paper we deduced the shifted Poisson structure on $\tilde R$.
We may understand them in this way:
geometrically, 
$\mathrm{DRep}_V(A)$ can be understood as
a DG subscheme
of 
$\mathrm{Rep}_V(\tilde R)$.
It is direct to see that the Poisson bracket \eqref{PoissononDRep}
on 
$\mathrm{Rep}_V(\tilde R)$
restricts to a Poisson bracket
on
$\mathrm{Rep}_V(R)$ studied in \cite{BCER} and \cite{CEEY};
in other words, $\mathrm{Rep}_V(R)$ is a shifted Poisson subscheme of
$\mathrm{Rep}_V(\tilde R)$.

\section{Example: The Sklyanin algebras}\label{sect:example}

In this section, we give two examples of Koszul Calabi-Yau algebras, namely
the 3- and 4-dimensional Sklyanin algebras,
and study the corresponding shifted bi-symplectic structure with some detail.
In general the geometry of the representations of Sklyanin algebras are complicated, 
however, both the derived representation schemes
$\mathrm{DRep}_V(A)$ and the representation schemes
$\mathrm{Rep}_V(\tilde R)$ are easy to describe; here
$A$ is a Sklyanin algebra and $\tilde R=\Omega(\tilde A^{\ac})$ as before.

\begin{example}[3-dimensional Sklyanin algebras]
Let $a,b,c\in k$ satisfying the following two conditions:
\begin{enumerate}
\item[$(1)$] $[a:b:c]\in\mathbb{P}^2_k\backslash D$, where
$$D=\{[1:0:0], [0:1:0], [0:0:1]\}\cup \{[a:b:c]|a^3=b^3=c^3=1\};$$
\item[$(2)$] $abc\ne 0$ and $(3abc)^3\ne (a^3+b^3+c^3)^3$.
\end{enumerate}
The 3-dimensional Sklyanin algebra $A=A(a,b,c)$
is the graded $k$-algebra with generators $x, y, z$ of degree one, and relations
$$
\begin{array}{c}
cx^2+bzy+ayz=0,\\
azx+cy^2+bxz=0,\\
byx+axy+cz^2=0.
\end{array}
$$
Smith showed in \cite[Example 10.1]{Smith}
that $A$ is Koszul, whose dual algebra $A^!$
is generated by $\xi_1,\xi_2,\xi_3$ with relations
$$\begin{array}{lll}
c\xi_2\xi_3-b\xi_3\xi_2,& b\xi_1^2-a\xi_2\xi_3,  & c\xi_3\xi_1-b\xi_1\xi_3,\\
b\xi_2^2-a\xi_3\xi_1,   & c\xi_1\xi_2-b\xi_2\xi_1,& b\xi_3^2-a\xi_1\xi_2.
\end{array}
$$
In degree 3, we have relations
$$
\begin{array}{l}
\xi_i^3=\displaystyle\frac{a}{b}\xi_1\xi_2\xi_3
=\displaystyle\frac{a}{b}\xi_3\xi_1\xi_2
=\displaystyle\frac{a}{b}\xi_2\xi_3\xi_1
=\displaystyle\frac{a}{c}\xi_2\xi_1\xi_3
=\displaystyle\frac{a}{c}\xi_3\xi_2\xi_1
=\displaystyle\frac{a}{c}\xi_1\xi_3\xi_2,\quad i=1,2,3,\\
\xi_i\xi_j^2=\xi_j^2\xi_i=0,\; i,j=1,2, 3\;\mbox{and}\; i\ne j.
\end{array}
$$
%
%
%
From these relations, we obtain that $A^{\ac}$ is isomorphic to a graded coalgebra
spanned by
\begin{equation}\label{generatorsofdualof3Skl}
\begin{array}{cl}
A_0^{\ac}: & {\mathbf 1}\\
A_1^{\ac}: & \xi_1^*,\xi_2^*,\xi_3^*\\
A_2^{\ac}: &\xi_1^*\xi_1^*,\xi_2^*\xi_2^*,\xi_3^*\xi_3^*\\
A_3^{\ac}:&\xi_3^*\xi_2^*\xi_1^*
\end{array}
\end{equation}
with the coproducts given by
\begin{eqnarray}
\Delta(\mathbf 1)&=&\mathbf 1\otimes\mathbf 1,\nonumber\\
\Delta(\xi_i)&=&\mathbf 1\otimes \xi_i+\otimes \xi_i\otimes \mathbf 1,\quad i=1,2,3,\nonumber\\
\Delta(\xi_1^*\xi_1^*)&=&\mathbf 1\otimes \xi_1^*\xi_1^* +\xi_1^*\otimes \xi_1^*
+\frac{a}{b}\xi_2^*\otimes\xi_3^*+\frac{a}{c}\xi_3^*\otimes\xi_2^*
+\xi_1^*\xi_1^*\otimes\mathbf 1,
\nonumber\\
\Delta(\xi_2^*\xi_2^*)&=&\mathbf 1\otimes \xi_2^*\xi_2^*+\xi_2^*\otimes\xi_2^*
+\frac{a}{c}\xi_1^*\otimes\xi_3^*
+\frac{a}{b}\xi_3^*\otimes\xi_1^*
+\xi_2^*\xi_2^*\otimes\mathbf 1,
\nonumber \\
\Delta(\xi_3^*\xi_3^*)&=& \mathbf 1\otimes \xi_3^*\xi_3^*+\xi_3^*\otimes\xi_3^*
+\frac{a}{b}\xi_1^*\otimes \xi_2^*
+\frac{a}{c}\xi_2^*\otimes \xi_1^*
+\xi_3^*\xi_3^*\otimes\mathbf 1,\nonumber\\
\Delta(\xi_3^*\xi_2^*\xi_1^*)
&=&{\mathbf 1}\otimes\xi_3^*\xi_2^*\xi_1^*
+\frac{b}{a}\xi_1^*\otimes\xi_1^*\xi_1^*
+\frac{b}{a}\xi_2^*\otimes\xi_2^*\xi_2^*
+\frac{b}{a}\xi_3^*\otimes\xi_3^*\xi_3^*
\nonumber\\
&&+\frac{b}{a}\xi_1^*\xi_1^*\otimes\xi_1^*
+\frac{b}{a}\xi_2^*\xi_2^*\otimes\xi_2^*
+\frac{b}{a}\xi_3^*\xi_3^*\otimes\xi_3^*
+\xi_3^*\xi_2^*\xi_1^*\otimes {\mathbf 1}.\label{coprodof3Skl}
\end{eqnarray}
The pairing on $A^!$ given by \eqref{coprodof3Skl} is cyclic, and therefore $A$ is 3-Calabi-Yau. 

Now, $\tilde R$ is the DG algebra freely generated by those elements in 
\eqref{generatorsofdualof3Skl} with degree shifted down by one.
From \eqref{coprodof3Skl} we see that
the $(-1)$-shifted bi-symplectic structure $\omega\in\mathrm{DR}^{2}_{\mathrm{nc}}\tilde R$
is given by
\begin{multline*}
\omega
=d({\mathbf 1})\otimes d(\xi_3^*\xi_2^*\xi_1^*)
+\frac{b}{a}d(\xi_1^*)\otimes d(\xi_1^*\xi_1^*)
+\frac{b}{a}d(\xi_2^*)\otimes d(\xi_2^*\xi_2^*)
+\frac{b}{a}d(\xi_3^*)\otimes d(\xi_3^*\xi_3^*)
\\
+\frac{b}{a}d(\xi_1^*\xi_1^*)\otimes d(\xi_1^*)
+\frac{b}{a}d(\xi_2^*\xi_2^*)\otimes d(\xi_2^*)
+\frac{b}{a}d(\xi_3^*\xi_3^*)\otimes d(\xi_3^*)
+d(\xi_3^*\xi_2^*\xi_1^*)\otimes d({\mathbf 1}).
\end{multline*}
Note that here (as well as in the next example) $d({\mathbf 1})\ne 0$.
\end{example}

\begin{example}[4-dimensional Sklyanin algebras]
Let $\alpha, \beta,\gamma\in k$ such that
$$\alpha+\beta+\gamma+\alpha\beta\gamma=0,\quad \{\alpha,\beta,\gamma\}\cap\{0,\pm 1\}=\emptyset.$$
The 4-dimensional Sklyanin algebra $A=A(\alpha,\beta,\gamma)$
is the graded $k$-algebra with generators $x_0,x_1,x_2,x_3$ of
degree one, and relations $f_i=0$, where
\begin{eqnarray*}
f_1=x_0x_1-x_1x_0-\alpha(x_2x_3+x_3x_2),&& f_2=x_0x_1+x_1x_0-(x_2x_3-x_3x_2),\\
f_3=x_0x_2-x_2x_0-\beta(x_3x_1+x_1x_3),&& f_4=x_0x_2+x_2x_0-(x_3x_1-x_1x_3),\\
f_5=x_0x_3-x_3x_0-\gamma(x_1x_2+x_2x_1),&& f_6=x_0x_3+x_3x_0-(x_1x_2-x_2x_1).
\end{eqnarray*}
As proved by Smith and Stafford (\cite[Propositions 4.3-4.9]{SS}),
$A$ is Koszul, whose Koszul dual algebra $A^!$
is generated by $\xi_0,\xi_1,\xi_2,\xi_3$ with the following relations:
\begin{eqnarray}
&&  \xi_0^2=\xi_1^2=\xi_2^2=\xi_3^2=0,\nonumber\\
&&  2\xi_2\xi_3+(\alpha+1)\xi_0\xi_1-(\alpha-1)\xi_1\xi_0=0,\nonumber\\
&&  2\xi_3\xi_2+(\alpha-1)\xi_0\xi_1-(\alpha+1)\xi_1\xi_0=0,\nonumber\\
&&  2\xi_3\xi_1+(\beta+1)\xi_0\xi_2-(\beta-1)\xi_2\xi_0=0,\nonumber\\
&&  2\xi_1\xi_3+(\beta-1)\xi_0\xi_2-(\beta+1)\xi_2\xi_0=0,\nonumber\\
&&  2\xi_1\xi_2+(\gamma+1)\xi_0\xi_3-(\gamma-1)\xi_3\xi_0=0,\nonumber\\
&&  2\xi_2\xi_1+(\gamma-1)\xi_0\xi_3-(\gamma+1)\xi_3\xi_0=0.\nonumber
\end{eqnarray}
Smith and Stafford also showed that $A^!$ admits a non-degenerate symmetric pairing,
and hence $A$ is 4-Calabi-Yau.
In particular, 
in degree 4, we have
the following identities:
\begin{eqnarray*}
&&\xi_0\xi_i\xi_0\xi_i=-\xi_i\xi_0\xi_i\xi_0\quad
\mbox{for}\; 1\le i\le j, \quad
\xi_0\xi_i\xi_0\xi_j=0\quad\mbox{for}\; i\ne j,\\
&&\xi_0\xi_2\xi_0\xi_2=\frac{1+\gamma}{1-\beta}\frac{1+\alpha}{1-\gamma}\xi_0\xi_1\xi_0\xi_1,\\
&&\xi_0\xi_3\xi_0\xi_3=\frac{1+\alpha}{1-\gamma}\xi_0\xi_1\xi_0\xi_1.
\end{eqnarray*}
From the above two groups of identities, we obtain that
the Koszul dual coalgebra $A^{\ac}$ is isomorphic to a coalgebra
spanned by
$$\begin{array}{cl}
A_0^{\ac}: & {\mathbf 1}\\
A_1^{\ac}: &\xi_0^*,\xi_1^*,\xi_2^*,\xi_3^*\\
A_2^{\ac}: &\xi_1^*\xi_0^*,\xi_2^*\xi_0^*,\xi_3^*\xi_0^*,\xi_0^*\xi_1^*,\xi_0^*\xi_2^*,\xi_0^*\xi_3^*\\
A_3^{\ac}:&\xi_0^*\xi_1^*\xi_0^*,\xi_0^*\xi_2^*\xi_0^*,\xi_0^*\xi_3^*\xi_0^*,\xi_1^*\xi_0^*\xi_1^*\\
A_4^{\ac}:&\xi_1^*\xi_0^*\xi_1^*\xi_0^*,
\end{array}
$$
with the coproduct of $\xi_1^*\xi_0^*\xi_1^*\xi_0^*$ given by
\begin{multline*}
\Delta(\xi_1^*\xi_0^*\xi_1^*\xi_0^*)
={\mathbf 1}\otimes \xi_1^*\xi_0^*\xi_1^*\xi_0^*
+\xi_0^*\otimes \xi_1^*\xi_0^*\xi_1^*
-\xi_1^*\otimes \xi_0^*\xi_1^*\xi_0^*\\
-\frac{1-\beta}{1+\gamma}\frac{1-\gamma}{1+\alpha}\xi_2^*\otimes \xi_0^*\xi_2^*\xi_0^*
-\frac{1-\gamma}{1+\alpha}\xi_3^*\otimes \xi_0^*\xi_3^*\xi_0^*
-\xi_1^*\xi_0^*\otimes \xi_1^*\xi_0^* \\
-\frac{1-\beta}{1+\gamma}\frac{1-\gamma}{1+\alpha}\xi_2^*\xi_0^*\otimes \xi_2^*\xi_0^* 
-\frac{1-\gamma}{1+\alpha}\xi_3^*\xi_0^*\otimes \xi_3^*\xi_0^* 
+\xi_0^*\xi_1^*\otimes \xi_0^*\xi_1^* \\
+\frac{1-\beta}{1+\gamma}\frac{1-\gamma}{1+\alpha}\xi_0^*\xi_2^*\otimes \xi_0^*\xi_2^* 
+\frac{1-\gamma}{1+\alpha}\xi_0^*\xi_3^*\otimes \xi_0^*\xi_3^* 
-\xi_1^*\xi_0^*\xi_1^*\otimes \xi_0^*\\
+\xi_0^*\xi_1^*\xi_0^*\otimes \xi_1^*
+\frac{1-\beta}{1+\gamma}\frac{1-\gamma}{1+\alpha} \xi_0^*\xi_2^*\xi_0^*\otimes \xi_2^*
+ \frac{1-\gamma}{1+\alpha}\xi_0^*\xi_3^*\xi_0^*\otimes \xi_3^*
+\xi_1^*\xi_0^*\xi_1^*\xi_0^*\otimes {\mathbf 1}.
\end{multline*}
Thus the $(-2)$-shifted bi-symplectic structure 
$\omega\in\mathrm{DR}^{2}_{\mathrm{nc}}\tilde R$
is given by
\begin{multline*}
d({\mathbf 1})\otimes d(\xi_1^*\xi_0^*\xi_1^*\xi_0^*)
+d(\xi_0^*)\otimes d(\xi_1^*\xi_0^*\xi_1^*)
-d(\xi_1^*)\otimes d(\xi_0^*\xi_1^*\xi_0^*)\\
-\frac{1-\beta}{1+\gamma}\frac{1-\gamma}{1+\alpha}d(\xi_2^*)\otimes d(\xi_0^*\xi_2^*\xi_0^*)
-\frac{1-\gamma}{1+\alpha}d(\xi_3^*)\otimes d(\xi_0^*\xi_3^*\xi_0^*)
-d(\xi_1^*\xi_0^*)\otimes d(\xi_1^*\xi_0^* )\\
-\frac{1-\beta}{1+\gamma}\frac{1-\gamma}{1+\alpha}d(\xi_2^*\xi_0^*)\otimes d(\xi_2^*\xi_0^* )
-\frac{1-\gamma}{1+\alpha}d(\xi_3^*\xi_0^*)\otimes d(\xi_3^*\xi_0^* )
+d(\xi_0^*\xi_1^*)\otimes d(\xi_0^*\xi_1^*) \\
+\frac{1-\beta}{1+\gamma}\frac{1-\gamma}{1+\alpha} d(\xi_0^*\xi_2^*)\otimes d(\xi_0^*\xi_2^* )
+\frac{1-\gamma}{1+\alpha}d(\xi_0^*\xi_3^*)\otimes d(\xi_0^*\xi_3^* )
-d(\xi_1^*\xi_0^*\xi_1^*)\otimes d(\xi_0^*)\\
+ d(\xi_0^*\xi_1^*\xi_0^*)\otimes d(\xi_1^*)
+ \frac{1-\beta}{1+\gamma}\frac{1-\gamma}{1+\alpha}d(\xi_0^*\xi_2^*\xi_0^*)\otimes d(\xi_2^*)
+\frac{1-\gamma}{1+\alpha} d(\xi_0^*\xi_3^*\xi_0^*)\otimes d(\xi_3^*)
\\
+d(\xi_1^*\xi_0^*\xi_1^*\xi_0^*)\otimes d({\mathbf 1}).
\end{multline*}
\end{example}


\begin{ack}
This work is partially supported by NSFC No. 11671281 and 11890663.
We also thank R. Nest and 
Song Yang for helpful comments during the preparation
of the work, and especially the anonymous referee for carefully reading the paper 
and pointing out many inaccuracies and typos.
\end{ack}


\begin{thebibliography}{100}


\bibitem{BCER}
Y. Berest, X. Chen, F. Eshmatov and A. Ramadoss,
{Noncommutative Poisson structures, derived representation schemes and Calabi-Yau algebras}.
Contemp. Math. {\bf 583} (2012), 219--246.

\bibitem{BFR}
Y. Berest, G. Felder and A. Ramadoss,
{Derived representation schemes and noncommutative geometry}.
Contemp. Math. {\bf 607} (2014), 113--162.


\bibitem{BKR}
Y. Berest, G. Khachatryan and A. Ramadoss,
{Derived representation schemes and cyclic homology}.
Adv. Math. {\bf 245} (2013), 625-689.

\bibitem{BR}
Y. Berest and A. Ramadoss,
{Stable representation homology and Koszul duality}.
J. Reine Angew. Math. {\bf 715} (2016), 143--187. 

\bibitem{BLS}R. Berger, T. Lambre and A. Solotar,
Koszul calculus, Glasgow Math. J. {\bf 60} (2018), no. 2, 361--399.

\bibitem{Bergman}G.M. Bergman, Coproducts and some universal ring constructions.
Trans. Amer. Math. Soc. {\bf 200} (1974), 33--88. 

\bibitem{BLB}R. Bocklandt and L. Le Bruyn,
Necklace Lie algebras and noncommutative symplectic geometry. Math. Z. 
{\bf 240} (2002), no. 1, 141--167.


\bibitem{CPTVV}D. Calaque, T. Pantev, B. To\"en and M. Vaqui\'e,
Shifted Poisson structures and deformation quantization.  
J. Topol. {\bf 10} (2017), no. 2, 483--584. 




\bibitem{CEEY}X. Chen, A. Eshmatov, F. Eshmatov and S. Yang,
The derived noncommutative Poisson bracket on Koszul Calabi-Yau algebras. 
J. Noncommut. Geom. {\bf 11} (2017), no. 1, 111--160. 

\bibitem{CEG}X. Chen, F. Eshmatov and W.L. Gan,
Quantization of the Lie bialgebra of string topology,
Communications in Mathematical Physics {\bf 301}:1(2011), 37--53. 

\bibitem{CYZ}X. Chen, S. Yang and G. Zhou,
Batalin-Vilkovisky algebras and the noncommutative Poincar\'e duality of Koszul Calabi-Yau algebras. 
J. Pure Appl. Algebra {\bf 220} (2016), no. 7, 2500--2532.

\bibitem{Cohen}P.M. Cohn,
The affine scheme of a general ring. Applications of sheaves 
(Proc. Res. Sympos. Appl. Sheaf Theory to Logic, Algebra and Anal., 
Univ. Durham, Durham, 1977), pp. 197--211, 
Lecture Notes in Math. {\bf 753}, Springer, Berlin, 1979. 

\bibitem{Connes}A. Connes, Noncommutative differential geometry. 
Inst. Hautes \'Etudes Sci. Publ. Math. No. {\bf 62} (1985), 257--360.

\bibitem{CB}W. Crawley-Boevey, Poisson structures on moduli spaces of representations.
J. Algebra {\bf 325} (2011), 205--215. 

\bibitem{CBEG}W. Crawley-Boevey, P. Etingof and V. Ginzburg, 
Noncommutative geometry and quiver algebras.
Adv. Math. {\bf 209} (2007), no. 1, 274--336. 

\bibitem{dTdVVdB}L. de Thanhoffer de V\"olcsey and M. Van den Bergh,
Calabi-Yau deformations and negative cyclic homology,
J. Noncommut. Geom. {\bf 12} (2018), no. 4, 1255--1291.






\bibitem{Ginzburg01}V. Ginzburg, Noncommutative symplectic geometry, quiver varieties, and operads.
Math. Res. Lett. {\bf 8} (2001), no. 3, 377--400.

\bibitem{Ginzburg}V. Ginzburg, Calabi-Yau algebras, arXiv:math/0612139.

\bibitem{GS}V. Ginzburg and T. Schedler, 
Moyal quantization and stable homology of necklace Lie algebras. 
Mosc. Math. J. {\bf 6} (2006), no. 3, 431--459, 587.

\bibitem{JM}J.D.S. Jones and J. McCleary,
Hochschild homology, cyclic homology, and the cobar construction,
Adams Memorial Symposium on Algebraic Topology,
Cambridge University Press (1992) 53--65.

\bibitem{Keller}B. Keller, Derived invariance of higher structures on the Hochschild complex,
available at https://webusers.imj-prg.fr/$~\sim$bernhard.keller/publ/index.html.

\bibitem{Keller1}B. Keller, Calabi-Yau triangulated categories. 
Trends in representation theory of algebras and related topics, 467--489, 
EMS Ser. Congr. Rep., Eur. Math. Soc., Z\"urich, 2008. 

\bibitem{KR}M. Kontsevich and A. Rosenberg,
Noncommutative smooth spaces. The Gelfand Mathematical Seminars, 1996--1999, 85--108, 
Gelfand Math. Sem., Birkh\"auser Boston, Boston, MA, 2000. 

\bibitem{Lam}Y. T. Lam, Calabi-Yau categories and quivers with superpotential,
Ph.D. Thesis 2015, available at http://people.maths.ox.ac.uk/joyce/theses/theses.html.

\bibitem{Loday}J.-L. Loday, Cyclic homology, 2nd edition,
Grundl. Math. Wiss. {\bf 301}, Springer-Verlag, Berlin, 1998.

\bibitem{LQ}J.-L. Loday, D. Quillen,
Cyclic homology and the Lie algebra homology of matrices,
Comment. Math. Helv. {\bf 59} (1984) 565--591.

\bibitem{LV}J.-L. Loday and B. Vallette, Algebraic operads.
Grundlehren der Mathematischen Wissenschaften {\bf 346}. Springer, Heidelberg, 2012. 

\bibitem{Melani}V. Melani,
Poisson bivectors and Poisson brackets on affine derived stacks. 
Adv. Math. {\bf 288} (2016), 1097--1120.

\bibitem{PTVV}T. Pantev, B. To\"en, M. Vaqui\'e and G. Vezzosi,
Shifted symplectic structures.
Publ. Math. Inst. Hautes \'Etudes Sci. {\bf 117} (2013), 271--328. 

\bibitem{Priddy}S.B. Priddy,
Koszul resolutions. Trans. Amer. Math. Soc. {\bf 152}, 1970, 39--60. 

\bibitem{Pridham}Pridham, Shifted Poisson and symplectic structures on derived N-stacks. 
J. Topol. {\bf 10} (2017), no. 1, 178--210. 




\bibitem{Quillen}D. Quillen, Algebra cochains and cyclic cohomology.
Inst. Hautes \'{E}tudes Sci. Publ. Math. {\bf 68} (1989), 139--174.

\bibitem{Schedler}T. Schedler, 
A Hopf algebra quantizing a necklace Lie algebra canonically associated to a quiver. 
Int. Math. Res. Not. 2005, no. {\bf 12}, 725--760. 

\bibitem{Shoikhet}B. Shoikhet, 
Koszul duality in deformation quantization and 
Tamarkin's approach to Kontsevich formality. Adv. Math. {\bf 224} (2010), no. 3, 731--771. 


\bibitem{Smith}
S.P. Smith, {Some finite dimensional algebras related to elliptic curves},
in {Representation Theory of Algebras and Related Topics} (Mexico City, 1994),
CMS Conf. Proc. 19, Amer. Math. Soc., Providence, 1996, 315--348.

\bibitem{SS}
S.P. Smith and J.T. Stafford,
{Regularity of the four dimensional Sklyanin algebra}.
Compositio Math. {\bf 83} (1992), 259--289.



\bibitem{VdB-1}M. Van den Bergh,
Noncommutative homology of some three-dimensional quantum spaces,
K-Theory {\bf 8} (1994) 213--230.

\bibitem{VdB0}M. Van den Bergh,
A relation between Hochschildt homology and cohomology for Georenstein rings,
Proc. Amer. Math. Soc. {\bf 126} (1998) 1345-1348,
and Erratum, Proc. Amer. Math. Soc. 130 (2002) 2809--2810.

\bibitem{VdB}M. Van den Bergh, Double Poisson algebras.
Trans. Amer. Math. Soc. {\bf 360} (2008), no. 11, 5711--5769. 

\bibitem{VdB2}M. Van den Bergh, Non-commutative quasi-Hamiltonian spaces.
Poisson geometry in mathematics and physics, 273--299, 
Contemp. Math. {\bf 450}, Amer. Math. Soc., Providence, RI, 2008.

\bibitem{VdB3}M. Van den Bergh, Calabi-Yau algebras and superpotentials.
Selecta Math. (N.S.) {\bf 21} (2015), no. 2, 555--603. 

\bibitem{Yeung}W.-K. Yeung, Pre-Calabi-Yau structures and the moduli of representations, arXiv:1802.05398v3.


\end{thebibliography}
\end{document}